\documentclass [11pt]{article}
\usepackage[]{amsmath,amssymb,amsfonts,latexsym,amsthm}
\usepackage[numeric,initials,nobysame]{amsrefs}
\usepackage[colorlinks=true, pdfstartview=FitV, pagebackref=false]{hyperref}
\usepackage{bbm,paralist}

\date{}
\title{Cores of random graphs are born Hamiltonian}
\author{{Michael Krivelevich\thanks{School of Mathematical
Sciences, Raymond and Beverly Sackler Faculty of Exact Sciences, Tel
Aviv University, Tel Aviv 69978, Israel. Email address: {\tt
krivelev@post.tau.ac.il}. Research supported in part by USA-Israel
BSF Grant 2010115 and by grants 1063/08, 912/12 from Israel Science
Foundation.}}
 \and {Eyal Lubetzky\thanks{Microsoft Research, One Microsoft Way, Redmond,
WA 98052-6399, USA. Email address: {\tt eyal@microsoft.com}.}}
\and {Benny Sudakov\thanks{Department of Mathematics, ETH, 8092 Zurich, Switzerland and Department of Mathematics, UCLA, Los Angeles, CA 90095. Email:
{\tt benjamin.sudakov@math.ethz.ch}. Research supported in part by NSF grant DMS-1101185, USA-Israeli BSF grant and by the Swiss National Science Foundation grant 200021-149111.}}
}
\numberwithin{equation}{section}
\newtheorem{maintheorem}{Theorem}

\newtheorem{theorem}{Theorem}[section]
\newtheorem{lemma}[theorem]{Lemma}
\newtheorem{claim}[theorem]{Claim}

\newtheorem{corollary}[theorem]{Corollary}

\renewcommand{\epsilon}{\varepsilon}

\newcommand{\E}{\mathbb{E}}
\renewcommand{\P}{\mathbb{P}}
\newcommand{\cC}{\mathcal{C}}

\newcommand{\cG}{\mathcal{G}}

\newcommand{\cS}{\mathcal{S}}

\newcommand{\cB}{\mathcal{B}}

\newcommand{\cJ}{\mathcal{J}}
\newcommand{\Gk}[1][G]{#1^{(k)}} 
\newcommand{\odd}{c_o}

\DeclareMathOperator{\bin}{Bin}

\newtheoremstyle{upright}%
        {8pt plus2pt minus4pt}%
        {8pt plus2pt minus4pt}%
        {\upshape}%
        {}%
        {\bfseries}%
        {:}%
        {1em}%
        {}%
\theoremstyle{upright}

\newcommand{\ignore}[1]{}

\begin{document}

\maketitle
\vspace{-0.5cm}
\begin{abstract}
Let $\{G_t\}_{t\geq0}$ be the random graph process ($G_0$ is
edgeless and $G_t$ is obtained by adding a uniformly distributed new
edge to $G_{t-1}$), and let $\tau_k$ denote the minimum time $t$
such that the $k$-core of $G_t$ (its unique maximal subgraph with
minimum degree at least $k$) is nonempty.  For any fixed $k\geq 3$
the $k$-core is known to emerge via a discontinuous phase
transition, where at time $t=\tau_k$ its size jumps from 0 to
linear in the number of vertices with high probability.
It is believed that for any $k\geq 3$ the core is Hamiltonian upon creation w.h.p., and Bollob\'as, Cooper, Fenner and Frieze further conjectured that it in fact admits $\lfloor(k-1)/2\rfloor$ edge-disjoint Hamilton cycles. However, even the asymptotic threshold for
Hamiltonicity of the $k$-core in $\cG(n,p)$ was unknown for any $k$.
We show here that for any fixed $k\ge 15$ the $k$-core of $G_t$ is
w.h.p.\ Hamiltonian for all $t \geq \tau_k$, i.e., immediately as
the $k$-core appears and indefinitely afterwards. Moreover, we prove
that for large enough fixed $k$ the $k$-core contains
$\lfloor (k-3)/2\rfloor$ edge-disjoint Hamilton cycles w.h.p.\ for
all $t\geq \tau_k$.

\end{abstract}

\section{Introduction}\label{sec:intro}

The fundamental problem of establishing the threshold for Hamiltonicity in the Erd\H{o}s-R\'enyi random graph traces back to the original pioneering works of Erd\H{o}s and R\'enyi introducing this model (see~\cite{ER60}). Following breakthrough papers by
P\'osa~\cite{Posa} 
and Korshunov~\cite{Korshunov}, this question has been
settled by Bollob\'as~\cite{Bollobas-1} and by Koml\'os and
Szemer\'edi~\cite{KS}, who proved that if $ p=(\log
n+\log\log n+\omega_n)/n$ for any diverging increasing sequence $\omega_n$ then the probability that a random
graph $G$ drawn from the probability space $\cG(n,p)$ of binomial
random graphs is Hamiltonian tends to $1$ as $n\to\infty$.
This threshold precisely coincides with the one for having minimum degree 2 in $\cG(n,p)$, and indeed a refined
hitting time version of this result was proved by Bollob\'as~\cite{Bollobas-1} and by
Ajtai, Koml\'os and Szemer\'edi~\cite{AKS}. They showed that in a typical random graph process $\{G_t\}_{t\geq 0}$ on $n$ vertices ($G_0$ is
edgeless and $G_t$ is obtained by adding a uniformly distributed
new edge to $G_{t-1}$), a Hamilton cycle appears \emph{exactly}
at the moment $t$ when the minimum degree of $G_t$ becomes 2 (an obvious prerequisite for Hamiltonicity).

With vertices of degree less than 2 being the primary obstacle for Hamiltonicity in the random graph, it was natural to instead consider the subgraph formed by (repeatedly) excluding such vertices. The $k$-\emph{core} of a graph $G$, denoted here by $\Gk$, is the (unique) maximal subgraph of $G$ of minimum degree at least $k$ (equivalently, it is the result of  repeatedly deleting vertices of degree less than $k$). This concept was introduced in 1984 by Bollob\'as~\cite{Bollobas-1}, who showed that for fixed $k\geq 3$ the $k$-core is typically nonempty and $k$-connected already at $p = C(k) / n$. {\L}uczak~\cite{Luczak87} established in 1987 that the threshold for the $k$-core to contain $\lfloor k/2 \rfloor $ edge-disjoint Hamilton cycles, plus a disjoint perfect matching for odd $k$, is when $p = (\frac1{k+1}\log n + k\log\log n)/n $.
In particular, the 2-core of $\cG(n,p)$ is Hamiltonian with high probability\footnote{An event is said to occur with high
probability (or w.h.p.\ for brevity) if its probability tends to 1 as $n\rightarrow\infty$.} when $p = (\frac13 \log n + 2\log\log n + \omega_n)/n$ for any diverging increasing sequence $\omega_n$, preceding the Hamiltonicity of the entire graph roughly by a factor of $3$.
However, for any fixed $k\geq 3$, the threshold for the $k$-core to contain a (single) Hamilton cycle remained open.

The best upper bound to date on this threshold is $(2+o_k(1))(k+1)^3/n$ by Bollob\'as, Cooper, Fenner and Frieze~\cite{BCFF}
via their study of the related model $\cG(n,m,k)$, a graph uniformly chosen out of all those on $n$ labeled vertices with $m$ edges and minimum degree at least $k$. It is well-known (and easily seen) that, conditioned on its number of vertices $N$ and edges $M$, the $k$-core of $\cG(n,p)$ is distributed as $\cG(N,M,k)$, hence the results on the latter model for a suitable $M=M(N)$ carry to $\Gk$ for $G\sim \cG(n,p)$. It was shown in~\cite{BCFF}, following a previous paper of Bollob\'as, Fenner and
Frieze~\cite{BFF}, that for any fixed $k\geq 3$ the graph $G\sim \cG(n,m,k)$ w.h.p.\ contains  $\lfloor
(k-1)/2 \rfloor$ edge-disjoint Hamilton cycles (plus a disjoint perfect matching when $k$ is even) as long as $m/n>(k+1)^3$
(see the recent work~\cite{Frieze} for an improvement in the special case $k=3$).
This can then be translated to $G\sim \cG(n,p)$ to imply that, for any fixed $k\geq 3$, its $k$-core is Hamiltonian w.h.p.\ provided that $p > (2+o_k(1))(k+1)^3/n$.

An obvious lower bound is the first appearance of the $k$-core, which on its own involves a remarkable first-order phase transition along the evolution of the random graph: Letting $k\geq 3$ and defining
\vspace{-0.09cm}
\begin{equation}\label{eq-tauk-def}
\tau_k = \min\big\{ t : G_t^{(k)} \neq \emptyset \big\}
\end{equation}
\vspace{-0.09cm}
to be the first time at which the random graph process gives rise to a nonempty $k$-core, {\L}uczak~\cites{Luczak91,Luczak92} showed that $|G_{\tau_k}^{(k)}| \geq n/5000$ w.h.p. That is, in a typical random graph process, the $k$-core remains empty until the addition of a single edge creates a chain reaction culminating in a linear-sized core. This fascinating phenomenon triggered a long line of works (see, e.g.,~\cites{CW,CM,Chvatal,Cooper,JL07,JL08,Molloy96,Molloy,PSW,PVW,Riordan}). Most notably, the tour-de-force by Pittel, Spencer and Wormald~\cite{PSW} determined the asymptotic threshold for its emergence to be $p=c_k/n$ for $c_k = \inf_{\lambda>0} \lambda / \P(\mathrm{Poisson}(\lambda) \geq k-1) = k + \sqrt{k\log k}+O(\sqrt{k})$, as well as characterized various properties of the core upon its creation, e.g., its size, edge density, etc.

\noindent\textbf{Main results.} The threshold for the Hamiltonicity of the $k$-core of $\cG(n,p)$ is roughly between $k/n$ and $2k^3/n$,
as discussed above. Comparing these, it was long believed that the lower bound should be asymptotically tight, and indeed Bollob\'as, Cooper, Fenner and Frieze~\cite{BCFF} conjectured (see also~\cite{PVW}) that their result could be pushed down to about $k/n$, and in particular, the $k$-core should be Hamiltonian w.h.p.\ for $p > c_k / n$.
 Furthermore, since the $k$-core is known to enjoy excellent expansion properties (being a random graph with a particularly nice degree distribution, namely, Poisson conditioned on being at least $k$) immediately upon its creation, one could expect the hitting time version of this result to also be true (see, e.g.,~\cite{JLR}*{Chapter 5} for related questions). That is, w.h.p.\ the $k$-core should be Hamiltonian since the very moment of its birth in the random graph process, i.e., at time $\tau_k$.
Here we are able to show this stronger statement for any $k\geq 15$, as given by the following theorem.

\begin{maintheorem}\label{thm-1}
Let $\{G_t\}$ be the random graph process on $n$ vertices. Fix $k\geq 15$ and let $\tau_k$
be the hitting time for the appearance of a $k$-core in $G_t$.
Then w.h.p.\ the $k$-core of $G_t$ is Hamiltonian for all $t \geq \tau_k$.
\end{maintheorem}

Note that the Hamiltonicity of the $k$-core is a \emph{non-monotone} property, and as such, guaranteeing that $G_t^{(k)}$ is Hamiltonian \emph{for all} $t \geq \tau_k$ is nontrivial even given that it is Hamiltonian at time $t=\tau_k$.
This is due to the fact that the $k$-core dynamically grows with the addition of edges along the random graph process. With each new vertex that joins the $k$-core, we must immediately find a new cycle going through this vertex as well as all the other vertices of the core.
This dynamic growth of the $k$-core poses additional challenges for the proofs, as we will now explain.

A common recipe for establishing Hamiltonicity in random graph models relies on \emph{sprinkling} new edges, through which longer and longer paths are formed in the graph, eventually culminating in a Hamilton path that is followed
by a Hamilton cycle. To show that the graph $G$ at time $T_1$ is Hamiltonian via this framework, one can examine the graph at time $T_0 = T_1 - c n$, call it $H$, and show that it enjoys the following property: Out of all edges missing from $H$, a constant proportion are ``good'' in the sense that adding any single good edge will support the above extension of the maximum path length. This is typically established using strong expansion properties of $H$, and in particular, this property is retained when additional edges are added to the graph. As such, when adding $c n$ edges (for an appropriate constant $c>0$) along the interval $(T_0,T_1)$, one can collect (w.h.p.) up to $n$ new good edges, leading to a Hamilton cycle. (A version of this framework also exists for a choice of $T_0=T_1 - o(n)$, provided that the initial graph $H$ satisfies additional more delicate properties.)

In our setting, the above framework faces a major obstacle due to the dynamical nature of the $k$-core, already present if one merely wishes to show the weaker asymptotic version of Theorem~\ref{thm-1} (i.e., that for any fixed $\epsilon>0$ the $k$-core of $G\sim\cG(n,p)$ is w.h.p.\ Hamiltonian at $p=(c_k+\epsilon)/n$).
Indeed, the very same sprinkling process that repeatedly extends our longest path simultaneously enlarges the $k$-core and might destroy the aforementioned properties of $H=G_{T_0}$ and foil the entire procedure.

Even more challenging is the obstacle in the way of a hitting time version (Hamiltonicity of $\Gk_t$ at $t=\tau_k$), due to the first-order phase transition by which the $k$-core emerges. Already at time $\tau_k - 1$ the $k$-core is empty, thus inspecting $H=G_{T_0}$ for $T_0 < T_1=\tau_k$ would not reveal which subgraph of $G_{T_1}$ will host the future core (and the sought properties will clearly not hold for every possible subgraph).

The novelty in this work is a different setup for sprinkling, hinging on \emph{conditioning on the future} at time $T_0$. Roughly put, at time $T_0$ we condition on the subset of vertices $V'$ that will host the $k$-core at time $T_1$, as well as on some of the edges along the interval $(T_0,T_1)$, in a careful manner such that:
\begin{compactenum}[\it(i)]
\item \label{it-J} A constant fraction of the edges along the interval $(T_0,T_1)$ remains unrevealed.
  \item\label{it-uniform} The unrevealed edges are \emph{uniformly distributed} over a (linear) subset of the edges that, all throughout the sprinkling process, contains most of the good edges.
\end{compactenum}
Observe that one can establish various properties of the induced graph of $H$ on $V'$ (e.g., expansion) and these would now be retained throughout the sprinkling (as $V'$ remains static in the new setup).
Properties~\eqref{it-J} and~\eqref{it-uniform} then enable one to appeal to the standard sprinkling framework for Hamiltonicity.
In our context, combining this method with the classical~\emph{rotation-extension} technique of P\'osa~\cite{Posa} reduces Hamiltonicity to the task of verifying certain properties of the induced graph on $V'$ at time $T_0$ (namely, suitable expansion and the existence of a sub-linear number of cycles covering all vertices). The latter task is where we require that $k \geq 15$ for the arguments to hold (and be reasonably short), and it is plausible that a more detailed analysis of this technical part would relax this restriction.

Engineering the above setup to begin with, however, is highly nontrivial:
Indeed, the identity of the vertex set of the future core $V'$ can give away a great deal of information on the upcoming edges until time $T_1$, easily biasing them away from being uniform on a given set of edges. Moreover, the set of good edges dynamically changes, and depends on the structure of the $k$-core in a complicated way...

The key step is to examine the set $V'$, conditioned to host the $k$-core at time $T_1$, and partition it into $\cB_k\subset V'$ and $\cC_k=V'\setminus \cB_k$, based on the graph $H$ at time $T_0=T_1-o(n)$ in the following manner:
The set $\cB_k$ will consist of every $v\in V'$ such that $v$ has at least $k$ neighbors in $V'$ already at time $T_0$. We then reveal all edges along the interval $(T_0,T_1)$ \emph{except} those with both endpoints in $\cB_k$ (giving away only the total number of such edges encountered along this interval) with the intuition being:
 \begin{compactenum}[(i)]
 \vspace{-0.09cm}
   \item Every $v\in \cB_k$ already guaranteed its inclusion in the $k$-core, thus the position of the unrevealed edges cannot interfere with our conditioning on $V'$, and consequently these edges are \emph{uniform}.
   \item As the $k$-core is linear and $T_1-T_0=o(n)$, we have $|\cB_k|=(1-o(1))|V'|$ and so almost all the good edges in our sprinkling process are going to have both of their endpoints lie within $\cB_k$.
 \end{compactenum}
 \vspace{-0.09cm}
Of course, turning this intuition into a rigorous argument requires a delicate analysis, working around several complications that conditioning on the future $V'$ entails. The full details are given in Section~\ref{sect2}.


Having settled the issue of Hamiltonicity in $k$-cores for
$k\geq 15$, we turn our attention to \emph{packing} Hamilton cycles,
addressing the conjecture of Bollob\'as, Cooper, Fenner and Frieze~\cite{BCFF} that the $k$-core should give rise to $\lfloor (k-1)/2\rfloor$ edge-disjoint Hamilton cycles (plus a perfect matching when $k$ is even).

As noted in~\cite{BCFF}, packing $\lfloor (k-1)/2\rfloor$ edge-disjoint Hamilton cycles in the $k$-core of $\cG(n,p)$ for $p=O(1/n)$ would be best possible w.h.p. The obstruction, as in the previous work of {\L}uczak~\cite{Luczak87}, is the presence of a so-called $k$-\emph{spider}, a vertex of degree $k+1$ whose neighbors all have degree $k$. (Whence, with $k$ being even, packing $k/2$ Hamilton cycles would clearly use all the edges incident to these degree-$k$ vertices, and in the process exhaust all $k+1$ edges incident to the root, impossible.)
As mentioned above, packing this many Hamilton cycles (possibly plus a matching) in the $k$-core was shown in~\cite{BCFF} to be possible roughly at $p=2 k^3 / n$, compared with its emergence threshold of $p\sim k/n$.

Here we nearly settle this conjecture, albeit for a sufficiently large $k$, by extracting $\lfloor
(k-3)/2 \rfloor$ edge-disjoint Hamilton cycles (one cycle short of the target number) as soon as the $k$-core is born. Similarly to the result in Theorem~\ref{thm-1}, this packing is then maintained indefinitely afterwards w.h.p.

\begin{maintheorem}\label{thm-2}
Let $\{G_t\}$ be the random graph process on $n$ vertices. Let
$\tau_k$ denote the hitting time for the appearance of a $k$-core in
$G_t$ for some large enough fixed $k$. Then w.h.p.\ the $k$-core of
$G_t$ contains $\lfloor (k-3)/2\rfloor$ edge-disjoint Hamilton
cycles for all $t \geq \tau_k$.
\end{maintheorem}

The proof of the above theorem follows the same framework of the proof of Theorem~\ref{thm-1}, with the main difference being the properties we require from the induced subgraph on the future $k$-core $V'$ before commencing the sprinkling process.
Instead of seeking a single 2-factor\footnote{A \emph{factor} is a spanning subgraph, and a \emph{$k$-factor} is one that is $k$-regular.}
in this subgraph as in the proof of Theorem~\ref{thm-1}, the new goal becomes finding a $\lfloor(k-3)/2\rfloor$ edge-disjoint 2-factors, which we can then convert to Hamilton cycles one at a time.
 To this end we rely on the recent results
of Pra{\l}at, Verstra{\"e}te and Wormald~\cite{PVW} and of Chan and
Molloy~\cite{CM}, which showed the existence of $(k-2)$-factors and $(k-1)$-factors in the $k$-core, respectively, for any sufficiently large $k$.


\noindent\textbf{Organization and notation.} 
In \S\ref{sect2} we prove Theorem~\ref{thm-1} modulo two technical
lemmas (Lemmas~\ref{lem-critical} and~\ref{lem-supercritical}, whose proofs appear in \S\ref{sec:core-properties})
on typical properties of $k$-cores in the critical and the
supercritical regimes, resp. 
The proof of Theorem~\ref{thm-2} appears in
\S\ref{sec:ham-pack}, and the final section \S\ref{sec:open} contains open problems.

Throughout the paper, all logarithms are in the natural basis.
 Our graph theoretic notation is standard: For a given graph $G$ and subsets $X,Y$ of its vertices, $e_G(X)$ and $e_G(X,Y)$ denote the numbers of
edges of $G$ spanned by $X$ and between $X$ and $Y$, respectively.
We let $N_G(X)$ denote the set of neighbors of $X$ in $G$, and occasionally refer to $N_G(X)\setminus X$ as the external
neighbors of $X$ in $G$.

\section{Hitting time for Hamiltonicity of the core}\label{sect2}

In this section we prove Theorem~\ref{thm-1} modulo two technical
lemmas whose proofs will be postponed to Section~\ref{sec:core-properties}. The task of establishing that the
$k$-core of almost every random graph process is Hamiltonian as soon
as it emerges and indefinitely afterwards is divided into two
separate regimes: an $O(n)$-interval of the random graph process
surrounding the abrupt emergence of the $k$-core, which we refer to
as the \emph{critical regime}, and the $O(n\log n)$-interval that
immediately follows it, referred to as the \emph{supercritical
regime}. Our aim in the critical regime is to establish
Hamiltonicity w.h.p.\ under the assumption that the $k$-core already
exists, whereas in the supercritical regime we aim to maintain the
Hamiltonicity property even as the $k$-core continues to grow.
Achieving these goals are the following
Theorems~\ref{thm-ham-small-c} and~\ref{thm-ham-large-c}, which
address the critical and supercritical regimes, respectively.
\begin{theorem}
  \label{thm-ham-small-c}
Fix $k\geq 15$ and let $G\sim \cG(n,m)$ for some $m$ satisfying $\max\{k/2,10\} n \leq m \leq \frac76 k n$.
Then with probability $1-o(n^{-3})$ either $\Gk$ is empty or it is Hamiltonian.
\end{theorem}

\begin{theorem}
  \label{thm-ham-large-c}
Fix $k\geq 15$, let $H\sim \cG(n,m)$ for $ \frac16(7k-1)n  \leq m \leq n\log n$, and let $G$ be the graph obtained by adding $\lfloor n/8\rfloor$ uniformly distributed new edges to $H$. Then the probability that $\Gk[H]$ is Hamiltonian, and at the same time $\Gk[G]$ is connected but not Hamiltonian, is at most $o(n^{-3})$.
\end{theorem}

Using well-known properties of the $k$-core, Theorem~\ref{thm-1} is
obtained as an immediate corollary of Theorems~\ref{thm-ham-small-c}
and~\ref{thm-ham-large-c}. Indeed, taking a union bound over the
interval $I_1=[\max\{k/2,10\} n,\, \frac76 k n]$ we can extend the
conclusion of Theorem~\ref{thm-ham-small-c} to hold with probability
$1-o(n^{-2})$ simultaneously for all $G_t$ with $t\in I_1$. Thus,
provided that $\tau_k$ falls within this interval, this implies
Hamiltonicity upon creation and up to $t \leq \frac76 kn$. The
explicit threshold in $\cG(n,p)$ for the sudden emergence of the
$k$-core for $k\geq 3$ is known (see,
e.g.,~\cites{Luczak91,PSW,PVW}) to be $c_k/n$ for $c_k = k +
\sqrt{k\log k } + O(\sqrt{k})$. Specifically in our
context, one can verify that $c_{15} \approx 20.98$ and $\max\{20,
k\} \leq c_k \leq 2 k$ holds for any $k\geq 15$. The corresponding
threshold $\tau_k$ for the emergence of the $k$-core along the
random graph process is concentrated around $c_k n /2$, and moreover
it is known (\cite{JL08}*{Theorem 1.4}) that $(\tau_k - c_k n
/2)/\sqrt{n}$ converges in distribution to a Gaussian
$\mathrm{N}(0,\sigma_k^2)$ for some (explicit) $\sigma_k^2>0$. In
particular, $\tau_k \in [\max\{k/2,10\}n,kn]$ except with
probability exponentially small in $n$, and altogether $G_t$ is
w.h.p.\ Hamiltonian for all $\tau_k \leq t \leq \frac76 kn$.

In order to maintain this property throughout the evolution of the
random graph process (in which the $k$-core grows and consequently
might not remain Hamiltonian), it suffices to do so up to $t = n\log
n$, as it is well-known that $\cG(n,m)$ for
$m=\frac{1+\epsilon}2n\log n$, for any constant $\epsilon>0$, is
w.h.p.\ $k$-connected and Hamiltonian (see,
e.g.,~\cites{Bollobas,JLR} and the references therein). A union
bound over the time interval $I_2=[\frac76k n - n/6,\, n\log n]$
extends the conclusion of Theorem~\ref{thm-ham-large-c} to hold with
probability $1-n^{-2+o(1)}$ simultaneously for all $G_t$ with $t\in
I_2$. With this probability we have that Hamiltonicity at time $t\in
I_2$ is carried to time $t+n/8$, unless the $k$-core at this later
time happens to be not connected. It was shown by {\L}uczak
(\cite{Luczak92}*{Theorem 5'}, see also~\cite{Luczak91}) that
w.h.p.\, throughout the entire random graph process $\{G_t\}$,
whenever the $k$-core is nonempty it is necessarily $k$-connected
and contains at least $n/5000$ vertices. It therefore follows that
Theorem~\ref{thm-ham-large-c} can extend the Hamiltonicity from an
interval $[a, a+n/8]$ for some $a\geq\frac76kn -n/6$ to hold w.h.p.\
for all $t\geq a$. As $\tau_k \leq kn$ w.h.p., with the above
conclusion from Theorem~\ref{thm-ham-small-c} we can take
$a=\lceil\frac76 kn -n/6\rceil$ and thus $G_t$ is w.h.p.\
Hamiltonian for all $t\geq\tau_k$, concluding the proof of
Theorem~\ref{thm-1}.

In the remainder of this section we prove
Theorems~\ref{thm-ham-small-c} and~\ref{thm-ham-large-c} modulo
Lemmas~\ref{lem-critical} and~\ref{lem-supercritical}, which are
stated next and whose proofs appear in Sections~\ref{subsec:gnp-near-tau} and~\ref{subsec:gnp-beyond-tau},
respectively.

\begin{lemma}\label{lem-critical}
Fix $k\geq15$ and let $G\sim \cG(n,m)$ for $m=c n/2$ such that
$\max\{k,20\} \leq c \leq \frac73 k$. With probability $1-o(n^{-3})$
the following holds for any $G'=(V',E')\subset G$ of minimum degree
$\delta(G')\geq k$ such that $|V'|\geq \frac45 n$ and the induced
subgraph of $G$ on $V'$ contains at most $|E'|+n/\log\log n$ edges:
\begin{compactenum}
  [(i)]
  \item \label{it-lem-1} the graph $G'$ is connected,
  \item \label{it-lem-2} every $X\subset V'$ of size $|X|\leq n/c^2$ satisfies $|N_{G'}(X)\setminus X| > 2|X|$, and
\item \label{it-lem-3} there exists a factor\footnote{By a slight abuse of notation, the factor described in Lemma~\ref{lem-critical} (which is a 2-factor minus a single edge) will also be referred to as a 2-factor for brevity.}
    in $G'$ consisting of a path $P$ and at most $n /\sqrt{\log n}$ cycles.
\end{compactenum}
\end{lemma}

\begin{lemma}
  \label{lem-supercritical}
Fix $k\geq15$, let $H\sim \cG(n,m)$ for $ \frac16(7k-1)n \leq m \leq n\log n$ and let $G$ be obtained by
adding $n/8$ uniformly chosen new edges to $H$. Let $G'=(V',E') \subset G$ be the graph where
$V'=V(\Gk)$ and $uv\in E'$ for $u,v\in V'$ if either $uv\in E(H)$ or alternatively
 $uv\in E(G)\setminus E(H)$ and at least one of $u,v$ belongs to $V' \setminus V(\Gk[H])$. Then with probability $1-o(n^{-3})$
we have $|V'| \geq |V(\Gk[H])|> 0.999 n$ and every subset $X \subset V'$ of size $|X|\leq n/5$ satisfies $|N_{G'}(X) \setminus X| > 2|X|$.
\end{lemma}

\begin{proof}[\textbf{\emph{Proof of Theorem~\ref{thm-ham-small-c}}}]
Given two graphs $H \subset G$ on the same vertex set $V$,
define the partition of the vertex set of $\Gk$, the $k$-core of $G$, into $\cB_k=\cB_k(G,H)$ and $\cC_k =\cC_k
(G,H)$ as follows:
\begin{align}
  \label{eq-N-S-def}
  \cB_k = \left\{ u\in V(\Gk)\;:\; \left|N_{H}(u) \cap V(\Gk)\right| \geq k \right\}\,,\qquad \cC_k = V(\Gk) \setminus \cB_k\,.
\end{align}
That is, $\cB_k$ consists of all vertices of $\Gk$ which possess at least $k$ neighbors within $V(\Gk)$ already in the subgraph $H$.
It is easily seen that $\cC_k$ is empty iff $V(\Gk[H])=V(\Gk)$, and that every vertex of $\cC_k$ must have
an edge incident to it in $E(G)\setminus E(H)$. We will be interested in settings where $\Gk$ is nonempty and
$|E(G)\setminus E(H)|=o(n)$, whence $\cB_k$ will comprise the bulk of the $k$-core of $G$.

Let $\cS=(e_1,\ldots,e_m)$ be a uniformly chosen ordered subset of $m$ edges out of the $\binom{n}2$ possible ones on the vertex set $V=[n]$,
and let $G$ and $H$ be the graphs on the vertex set $V$ with edge sets
\begin{align*}
  E(G)&=\{e_1,\ldots,e_m\}\,,\\
  E(H)&=\{e_1,\ldots,e_{m'}\}\qquad\mbox{ where }\qquad m' = m - \frac{n}{\log\log n}\,.
\end{align*}
As $G\sim\cG(n,m)$, if $\Gk$ is empty then there is nothing left to prove. We therefore assume otherwise and aim to establish
that $\Gk$ is Hamiltonian with probability $1-o(n^{-3})$.
The key to the proof will be to examine the random variables $\cB_k(G,H)$ and $\cC_k(G,H)$ as defined above.
Already we see that $|\cC_k| \leq 2(m-m') = 2n/\log\log n$, which would make it negligible compared to the linear size of $\Gk$.

Denote the indices of the subset of the final $m-m'$ edges having both endpoints in $\cB_k$ by
\begin{equation}
  \label{eq-cJ-def}
  \cJ = \left\{ m' < j \leq m \;:\: e_j = (u,v)\mbox{ for some }u,v\in\cB_k\right\}\,,
\end{equation}
and let $J = |\cJ|$.
We claim that upon conditioning on $\cB_k$, $\cC_k$ (such that $\cB_k\cup\cC_k \neq\emptyset$ by our assumption) as well as $\cJ$ and all the edges $\cS^* = \{ e_j : j \notin \cJ\}$, the remaining $J$ edges of $\cS$ are uniformly distributed over all edges missing from $\cB_k$ (that is, edges with both endpoints in $\cB_k$ that did not appear among $e_1,\ldots,e_{m'}$). To establish this it suffices to show that adding any possible set of $J$ new edges with endpoints inside $\cB_k$ to $\cS^*$ would result in a graph such that $\cB_k \cup \cC_k$ is consistent with our conditioning (since each such extension would have the same probability, and the partition $V'=\cB_k\cup\cC_k$ only depends on the graph $H$ and the vertex set $V'$).

Indeed, letting $F=(V,\cS^*)$ we argue that $V(\Gk[F])=V(\Gk)$. The fact that $V(\Gk[F])\subset V(\Gk)$ follows trivially from having $\cS^* \subset \cS$.
At the same time, $V(\Gk)\subset V(\Gk[F])$ since every $u\in \cC_k(G,H)$ has $N_F(u) = N_G(u)$ (which contains, in particular, at least $k$ vertices from $V(\Gk)$) whereas every $u\in\cB_k(G,H)$ has $ N_H(u) \subset N_F(u)$ and $\left|N_H(u)\cap V(\Gk)\right|\geq k$
by the definition of $\cB_k(G,H)$. It is now clear for any subset $S$ of edges whose endpoints all lie within $\cB_k(G,H)$ (or even within $V(\Gk)$ for that matter), the vertex set of the $k$-core of $F'=(V,\cS^* \cup S)$ remains equal to $V(\Gk)$. That is, every extension of $\cS^*$ via $J$ new internal edges of $\cB_k$ is consistent with our conditioning, as claimed.

A delicate point worth noting is that the above statement would be false if we were to directly take $m=\tau_k$, the hitting time for the emergence of the $k$-core. In that case, it could occur that adding certain edges between the vertices of $\cB_k$ to the graph $F=(V,\cS^*)$ would induce a different $k$-core to appear at some earlier time $j < m$, thus violating our conditioning. This issue disappears once we fix $m$ deterministically, thereby guaranteeing that no edges within $\cB_k$ would be forbidden.

The next crucial step is controlling the random variable $J$, which counts the number of edges within $\cB_k$ between times $m'+1,\ldots,m$.
Our goal is to show that $J / (m-m')$ is uniformly bounded from below, that is, that a constant fraction of the edges $e_{m'+1}\,\ldots,e_m$ are internal edges of $\cB_k$. Intuitively, one could hope that the distribution of the edges in $E(G)\setminus E(H)$ would be close to uniform. In that case, since the graphs $G,H$ under consideration are such that $\cB_k$ is of linear size,
standard concentration arguments would yield the sought estimate except with a probability that is exponentially small in $n$.
While such an argument would be valid for a \emph{fixed} set of vertices in $H$, unfortunately the set $\cB_k$ is random and \emph{does} depend on $e_{m'+1},\ldots,e_m$.

For instance, consider the situation where conditioned on $H$ and
then on the partition $\cB_k \cup \cC_k$ we have that $\cC_k =
2(m-m')$. This occurs when the edges $e_{m'+1},\ldots,e_m$ form a
matching on the vertices of $\cC_k$ (so as to accommodate the degree
constraints of all $2(m-m')$ of its vertices). Rare as this event
may be, it demonstrates the possibility that $J=0$ (even
deterministically) in this conditional space. Another example of the
delicate dependence between $J$ and $\cB_k$ is the following:
Suppose that instead of $\cB_k$ we would work with its variant
$\cB'_k$, consisting of all vertices of $V(\Gk)$ whose set of
neighbors within $V(\Gk)$ is the same in $G$ as it is in $H$. Then
$\cB'_k \subset \cB_k$ and these two sets differ by at most
$m-m'=o(n)$ vertices, and yet the variable $J'$ corresponding to
$\cB'_k$ satisfies $J'=0$ by definition.

As the next claim shows, it is possible to control the random variable $J$ despite the dependence between the $\cB_k$ and $E(G)\setminus E(H)$.

\begin{claim}
  Let $J=|\cJ|$ for $\cJ$ as defined in~\eqref{eq-cJ-def}. Then
  \[ \P\left( J \geq n/(100 \log\log n)\right) = 1-O(n^{-9})\,.\]
\end{claim}
\begin{proof}
The idea underlying this approach would be to approximate $\cB_k$ via another random set whose identity is completely determined by $G$ (as opposed to the combination
of $G$ and $H$).

Let $A_1 = A_1(G)$ denote the event that there is a subset $S$ of $s=\lfloor \frac25 n\rfloor$ vertices
  of $G$ which spans at most $r = m/50$ edges.
  Taking $p=c/n$ for $c=2m/n$ we appeal to the standard correspondence between $\cG(n,m)$ and $\cG(n,p)$, and
  since $r \sim \frac1{16}s^2 p$ we get that
  \begin{align*}
 \P(A_1) &\leq (r+1) \binom{n}s \binom{\binom{s}2 }{r} p^r (1-p)^{\binom{s}2-r}
 \leq \left((1+o(1))\frac{en}{s} \left(\frac{e s^2 p}{2r}\right)^{r/s} e^{-sp/2} \right)^s \\
 &\leq \left(\left(\tfrac52 e + o(1)\right) (8 e)^{c/40} e^{- c/5} \right)^s =
 \left(\left(\tfrac52 e + o(1)\right) (8 e^{-7})^{c/40}  \right)^s < (3/4)^s\,,
  \end{align*}
where the last inequality uses our hypothesis that $c\geq 20$. It follows that the probability
  of encountering $A_1$ is at most $O(\exp(-a_0 n))$ for some absolute $a_0>0$.

Let $X_l=X_l(G)$ denote the number of vertices of degree precisely $l$ in $\Gk$ for a fixed integer $l\geq k$.
It is known (\cite{CW}*{Corollaries~2 and~3}) that the fraction of vertices of degree $l$ within the $k$-core converges to $\P(Z=l)$ where $Z$ is a Poisson random variable with mean $\mu$ for some explicitly given $\mu=\mu(k)>k$.
More precisely, for any fixed $\epsilon>0$,
\begin{equation}
  \label{eq-core-degrees}
\P\left(\left| X_l - \frac{e^{-\mu} \mu^l}{l!}n\right|>\epsilon n\right) = O\left(e^{-n^{a_1} }\right)\,,
\end{equation}
where $a_1=a_1(k) >0$ is fixed.
One should note that~\cite{CW} offers (sharper) estimates only for the supercritical range $c>c_k$. However, one may extend these to $c \sim c_k$, given that the $k$-core is nonempty at that point, as follows. Fixing some arbitrarily small $\delta >0$, at $c'=c_k + \delta$ these bounds yield concentration for $X_l$ within a window of width $\epsilon n$ (for any fixed $\epsilon>0$) except with an exponentially small probability. Next, appealing to the known estimates on $|V(\Gk)|$ and $|E(\Gk)|$ throughout the critical window for the emergence of the $k$-core (see~\cite{JL08}*{Theorem 1.3}), it is known that the difference in these two random variables between times $c n/2$ and $c' n/2$ is at most some $\epsilon n$ except with probability $O(e^{-n^{a_1}})$ for some $a_1>0$ fixed, where $\epsilon$ can be made arbitrarily small via selecting a suitably small $\delta$. Consequently, the value of $X_l$
cannot change by more than $2\epsilon n$ along this interval (accounting for lost edges as well as lost vertices via the change in $V(\Gk)$ and $E(\Gk)$, respectively), thus
establishing~\eqref{eq-core-degrees} with a window of $3\epsilon n$. To complement this bound, observe that for fixed $l\geq 15$,
\[\frac{e^{-\mu} \mu^l}{l!} \leq \frac{e^{-\mu} \mu^l}{\sqrt{2\pi l}(l/e)^l} = \frac{e^{-(\mu-l)}(1+\frac{\mu-l}l)^l}{\sqrt{2\pi l}} \leq \frac1{\sqrt{2\pi l}} < \frac18\,,\]
where the first inequality above followed from Stirling's formula and the last one holds for $l\geq 15$. Altogether, we can infer that the number of vertices whose degree belongs to $\{k,k+1,k+2\}$ satisfies
\begin{align*}
    \P\left( X_k + X_{k+1} + X_{k+2} \leq \tfrac25 n\right) = 1-O\left(e^{ - a_1 n}\right)\,.
 \end{align*}
At the same time, known properties of the $k$-core upon its emergence (recall that $\Gk$ is nonempty by assumption), namely the explicit formula for its typical initial size as well its concentration around its mean (see, e.g.,~\cite{PSW}*{Theorem 3}), imply that for some absolute $a_2>0$ and any fixed $k\geq 15$,
\begin{align*}
 \P\left(|V(\Gk)| \geq \tfrac45 n \right) = 1 - O\left(e^{- n^{a_2}}\right)\,.
 \end{align*}
Let $A_2=A_2(G)$ denote
the event that $X_k+X_{k+1}+X_{k+2}> \frac25 n$ or $|V(\Gk)| <
\frac45 n$.
The last two estimates then imply that $\P(A_2) \leq O(e^{-n^{a_2}})$,
and on the event $A_2^c$ we see that at least $\frac25 n$ vertices have
degree at least $k+3$ in $\Gk$.

A final ingredient we need is some control over vertices of large degree in $G$. Set
\begin{align}\label{eq-KRM-def}
 \kappa_1 = \log\log n\,,\qquad \kappa_2 = 10\frac{\log n}{\log\log n}\,,\qquad \rho = \frac{n}{\log^2 n}\,,
 \end{align}
and let $A_3=A_3(G)$ denote the event that either $\Delta(G) \geq \kappa_2$, where $\Delta(G)$ denotes the maximum degree of $G$, or there are $\rho$ vertices of $G$ whose degree exceeds $\kappa_1$.
It is well-known that whenever $m/n$ is uniformly bounded from above, the maximum degree in $\cG(n,m)$ is at most $(1+o(1))\frac{\log n}{\log\log n}$ w.h.p.\ (see,
e.g., \cite{Bollobas}*{\S3}), and moreover, the probability this maximum degree would exceed $\kappa_2$ as given in~\eqref{eq-KRM-def} is at most $O(n^{-9})$.
Turning to the probability that a given $v\in V$ has degree at least $\kappa_1$ in $G$, again working with the corresponding $\cG(n,p=c/n)$ model with $c=2m/n$ we find it to be
\[ \P(\bin(n-1,p)\geq \kappa_1) =  O\left(c^{\kappa_1} / {\kappa_1}!\right) = e^{-(1-o(1))\kappa_1 \log \kappa_1} < e^{-5\kappa_1} = (\log n)^{-5}\,,\]
where the strict inequality holds for large $n$.
We will now argue that the probability of encountering $\rho$ such vertices (recall that  $\rho= n^{1-o(1)}$) would be $\exp(-n^{1-o(1)})$. Indeed, the probability of encountering a set $S$ of $\rho$ vertices such that the induced subgraph on $S$ contains more than $10\rho$ edges is at most
\[
\binom{n}{\rho} \binom{\binom{\rho}{2}}{10\rho} (c/n)^{10\rho} \leq
\left( \left(\frac{e n}{\rho}\right)^{\frac1{10}} \frac{e \rho c}{20 n} \right)^{10\rho} = \left( \left(e \log^2 n\right)^{\frac1{10}}  \frac{e  c}{20 \log^2 n} \right)^{10\rho} = e^{-n^{1-o(1)}}\,.
\]
At the same time, in a set $S$ of $\rho$ vertices whose degrees are all at least $\kappa_1$ and where additionally the induced subgraph on $S$ has at most $10\rho$ edges we must have at least $\rho (\kappa_1 - 20)$ edges in the cut between $S$ and $V(G)\setminus S$. Similarly to the previous calculation, the probability of this event is at most
\[ \binom{n}{\rho} \binom{\rho(n-\rho)}{\rho(\kappa_1-20)} (c/n)^{\rho(\kappa_1-20)} =
\left(e \log^2 n \left(\frac{e (n-\rho)c}{(\kappa_1-20)n}\right)^{\kappa_1-20} \right)^{\rho}
= e^{-(1-o(1))\rho \kappa_1\log\kappa_1} = e^{-n^{1-o(1)}}\,.
\]
Altogether we deduce that $A_3$ occurs with probability at most $O(n^{-9})$ with room to spare.

Let $U\subset V(\Gk)$ be the set of all vertices whose degree in $\Gk$ is at least $k+3$.
It is important to note that $U$ as well as the properties addressed by the events $A_1,A_2,A_3$ are entirely determined by the edge set $E(G)=\{e_1,\ldots,e_m\}$ regardless of the order in which they appeared (as opposed to $\cB_k,\cC_k$ which were a function of $G$ and $H$).
Condition on this edge set $E(G)$, unordered, and further condition on the event $A^c =A_1^c \cap A_2^c \cap A_3^c$, recalling that $\P(A^c) = 1-O(n^{-9})$
by the above analysis.

Under this conditioning, the edge set $E(H)$ is obtained as a uniform subset of $m'$ of these edges, hence the variable
\[ J_U = \#\{ uv \in E(G)\setminus E(H) \;:\; u,v\in U\}\]
is hypergeometric. Specifically, $J_U$ is the result of $m-m'=n/\log\log n$ samples without repetition with the target population having proportion at least $1/50$, since $A_1^c$ guarantees at least $m/50$ internal edges within any set of $\frac25 n$ vertices whereas $A_2^c$ implies that $|U|\geq \frac25 n$.
Hoeffding's inequality for hypergeometric variables~\cite{Hoeffding} now tells us that the probability that for some absolute constant $a_3>0$,
\[ \P( J_U < (m-m')/75\mid E(G)\,,\,A^c) \leq e^{-a_3 (m-m')} = e^{-n^{1-o(1)}}\,,\]
where the last inequality holds for large enough $n$.

Now let $W_1\subset U$ denote the subset of vertices of $U$ which have degree at most $\kappa_1$ in $G$ and are incident to at least $4$ edges of $E(G)\setminus E(H)$.
The number of edges in the sample $E(G)\setminus E(H)$ that are incident to a fixed set of $s$ such vertices is stochastically dominated by a binomial variable $\bin(m-m', s \kappa_1 / m')$, as there are overall at most $s \kappa_1$ edges incident to this set in $G$. Hence,
\begin{align*}
\P(|W_1| \geq \rho\mid E(G)\,,\, A^c) &\leq \P\left(\bin\left(m-m', \rho \kappa_1/m'\right)\geq 2\rho\right) \leq \binom{m-m'}{2\rho}\left(\frac{\rho \kappa_1}{m'}\right)^{2\rho} \\
&\leq \left( \frac{e n}{2\rho \log\log n} \frac{\rho \log\log n}{(10-o(1)) n}\right)^{2\rho} = \left(\frac{e}{20}+o(1)\right)^{2\rho} = e^{-n^{1-o(1)}}\,,
 \end{align*}
 where the first inequality in the second line used the hypothesis $m' \sim m \sim cn/2 \geq 10 n$.
Conditioned on this event, our assumption on the maximal degree of $G$ implies that the number of edges incident to $W_1$ in $G$ is at most $\kappa_2 \rho = o(n/\log
n)$.
Similarly, letting $W_2\subset U$ be the set of vertices whose degree in $G$ exceeds $\kappa_1$, we know by assumption that $|W_2|\leq \rho$ and that the number of
edges incident to this set in $G$ is then at most $\kappa_2 \rho = o(n/\log n)$.

Letting $U^* = U \setminus (W_1\cup W_2)$, by definition we have that every vertex of $U^*$ is incident to at most $3$ edges of $E(G)\setminus E(H)$.
At the same time, the degree in $\Gk$ of every $v\in U$ is at least $k+3$, and consequently each $v\in U^*$ has at least $k$ neighbors among $V(\Gk)$ in $H$,
hence $U^* \subset \cB_k$. The above estimates show that, except with probability $\exp(-n^{1-o(1)})$, there are at least $n/(75\log\log n)$ edges with both endpoints in $U$ (counted by the variable $J_U$ above) whereas the total number of edges incident to $W_1 \cup W_2 = U\setminus U^*$
is $o(n/\log n) = o(n/\log\log n)$. Therefore,
\[ \P\left( J < n/(100 \log\log n)\mid E(G)\,,\,A^c\right) < e^{-n^{1-o(1)}}\,,\]
and the desired (unconditional) estimate on $J$ now follows from the fact that $\P(A^c)=1-O(n^{-9})$.
\end{proof}

We can now turn to Lemma~\ref{lem-critical}, taking the target graph $G'$ to be $\Gk[F]$ for $F=(V,\cS^*)$.
As mentioned above, it is known that whenever the $k$-core for $k\geq 15$ is nonempty, its size is at least $\frac45 n$ except with probability exponentially small in $n$. Recalling that $V(\Gk[F])=V(\Gk)$ and that $E(G)\setminus E(F)$ contains $J \leq n/\log\log n$ internal edges of $\cB_k$, the requirements of the lemma are satisfied.
This lemma will be used in tandem with the following well-known corollary of the classical rotation-extension technique of P{\'o}sa \cite{Posa} (see, e.g.,~\cite{KLS}*{Corollary 2.10}, which is formulated slightly differently though the exact same proof implies both statements), where here and in what follows the length of paths/cycles will refer to the number of edges in them:
\begin{lemma}\label{lem-Posa}
Let $r$ be a positive integer, and let $G=(V,E)$ be a graph where for every subset $R\subset V$ with $|R|<r$ we have $|N_G(R)\setminus R|\ge 2|R|$.
Then for any path $P$ in $G$, denoting its length by $h$,
\begin{compactenum}
  [(i)]
  \item\label{it-Posa-longer-path} there is path $P'$ of length $h+1$ in $G$ containing all vertices of $P$ plus a new endpoint $u\notin P$, or
  \item\label{it-Posa-closing-cycle} there is a cycle $C'$ in $G$ of length $h+1$ on the same vertex set as $P$, or
  \item\label{it-Posa-boosters} there are at least $r^2/2$ non-edges in $G$ such that if any of them is turned into an edge, then the new graph contains a cycle $C'$ as above, i.e., an $(h+1)$-cycle on the same vertex set as $P$.
\end{compactenum}
\end{lemma}

Lemma~\ref{lem-critical} enables us to apply the rotation-extension
technique to $G'$ with $r=n/c^{2}$. Let $P$ be the path in $G'$ such that $V(G')\setminus P$ can be covered by $\ell \leq n/\sqrt{\log n}$ cycles as per the conclusions of Lemma~\ref{lem-critical}. Let $h$ denote the length of $P$ (i.e., $P$ has $h$ edges) and let $\{C_1,\ldots,C_\ell\}$ denote the cycles covering $V(G')\setminus P$.  We now argue that either $G'$ is Hamiltonian or there exist at least $r^2/2$ edges --- to be referred to as \emph{boosters} --- the addition of each of which would create either a Hamilton cycle or a path $P'$ containing all vertices of $P$ in addition to one of the $C_i$'s. Thereafter, we will sprinkle the edges $e_{i_1},\ldots,e_{i_J}$ to repeatedly absorb all $\ell$ cycles into $P$ and then form a Hamilton cycle.

To justify the above claim, examine the three possible conclusions of Lemma~\ref{lem-Posa} above:
\begin{compactenum}[(i)]
  \item If the path $P$ can be extended to a path $P'$ ending at some vertex $u \in V(G')\setminus P$, delete an edge to unravel the cycle $C_i$ containing $u$ into a path, and concatenate that path to $P'$.
  \item If there is a cycle $C'$ on the same vertex set as $P$, either $C'$ is a Hamilton cycle in $G'$ as required, or the connectivity of $G'$ would imply that $C'$ is connected by a path to some $u\in V(G')\setminus P$. In the latter case, unravel both $C'$ and the cycle $C_i$ of $u$ in the obvious way into a single path.
  \item Otherwise, there are at least $r^2/2$ edges, the addition of each of which would lead to Case~\eqref{it-Posa-closing-cycle}.
\end{compactenum}

Overall, letting $G'_0=G'$ and $G'_j$ be the result of adding the edge $e_{i_j}$ to $G'_{j-1}$ for $j=1,\ldots,J$, we observe that
the above mentioned properties of $G'$ are satisfied for each $j$ (being closed under the addition of edges) and so encountering $\ell + 1 \leq n/\sqrt{\log n}+1$ rounds
for which $e_{i_j}$ is a booster in $G'_{j-1}$ would guarantee Hamiltonicity.

The proof can now be concluded from arguing that, for any $j=1,\ldots,J$, the probability that $e_{i_j}$ is a booster in $G'_{j-1}$ is uniformly bounded from below. Indeed, $e_{i_j}$ is uniformly distributed over all missing edges in the induced subgraph of $G'_{j-1}$ on $\cB_k$. Of course,
$V(G'_{j-1}) = V(\Gk) = \cB_k \cup \cC_k$ so there are $o(n^2)$ edges incident to $\cC_k$ in $G'_{j-1}$. It follows
that almost all of the boosters of $G'_{j-1}$ are edges whose both endpoints lie in $\cB_k$, and the probability that $e_{i_j}$ belongs to this set is at least $(1-o(1))(r^2 /2) / \binom{|\cB_k|}2 \geq (r/n)^2 \geq 1/c^4$.
Hence, conditioned on the event $J \geq n/(100\log\log n)$, the number of boosters we collect (formally defining every new edge as a booster upon achieving Hamiltonicity)
stochastically dominates the random variable $Y\sim \bin(n/(100\log\log n), 1/c^4)$, which in turn satisfies $\P(Y \geq c' n/\log\log n) \geq 1 - \exp(-n^{1-o(1)})$ for some $c'>0$ (that depends on $c$). In particular, $Y > n/\sqrt{\log n} + 1$ with probability $1-\exp(-n^{1-o(1)})$ and the proof is complete.
\end{proof}

\begin{proof} [\textbf{\emph{Proof of Theorem~\ref{thm-ham-large-c}}}]
The proof will follow the same line of argument used to prove Theorem~\ref{thm-ham-small-c}, albeit in a somewhat simpler setting as we can utilize the Hamilton cycle in $\Gk[H]$ as
the foundation of our 2-factor, and furthermore, we can afford to select the subset $\cB_k$ of vertices --- the bulk of the $k$-core to be used for sprinkling --- based on the
information of $H$ (in lieu of conditioning on future information from the random graph process). Namely, we will choose $\cB_k$ to be the vertex set of $\Gk[H]$. Recall that a key
element in the proof of Theorem~\ref{thm-ham-small-c} revolved around controlling the random variable $J$ despite its delicate dependence on $\cB_k$. As mentioned there, this issue
would be circumvented should $\cB_k$ be determined already by graph $H$, precisely our present situation.

Let $H\sim\cG(n,m)$ for $m\geq \frac16(7k-1)n$. Ignoring floors and ceilings for brevity here and throughout this proof, let $e_1,\ldots,e_{n/8}$
be a uniformly chosen subset of the edges missing from $E(H)$, and let $G$ be the result of adding
these edges to $H$.

Suppose that $\Gk[H]$ is Hamiltonian (otherwise there is nothing to prove), let $\cB_k= V(\Gk[H])$, and observe that thanks to Lemma~\ref{lem-supercritical} we already
know that $|\cB_k| > 0.999 n$ with probability $1-o(n^{-3})$.
To exploit the other conclusions of that lemma, we must first recover the graph $G'$ as defined there.

As before, we will condition on the graph $G'$ as well as on the random variable $J$ counting the number of
edges with both endpoints in $\cB_k$, without revealing the identity of those edges.
As we mentioned, the fact that $\cB_k$ is determined by $H$ will readily yield the sought lower bound on $J$.
Formally, for each $t=1,\ldots,n/8$, expose whether the new edge $e_t=uv$ (uniformly distributed out of all missing ones)
has either $u\notin \cB_k$ or $v\notin \cB_k$, and if so, expose its endpoints $u,v$ themselves.
This process reveals the graph $G'$ and random variable $J$, yet we see that each edge is uniformly distributed out of all missing ones
(as we no longer have a conditioning involving the graph process at future times as in the proof of Theorem~\ref{thm-ham-small-c}).
At all times we have at most $m \leq n\log n + n/8$ edges in our graph (as per the scope of Theorem~\ref{thm-ham-large-c}), therefore
the probability that $e_t$ has both its endpoints in $\cB_k$ is at least $\big(\binom{0.999 n}2 - 2n\log n \big)/\binom{n}2 > 0.998$,
with the last inequality being valid for large enough $n$. That is to say, the random variable $J$ counting the number of such edges that we encounter)
stochastically dominates a binomial variable $Y \sim \bin(n/8, 0.998)$, whence standard concentration arguments imply that $J \geq n/10$ except with probability exponentially small in
$n$. Finally, by the definition of the process above, each edge $e_t$ that was accounted for in $J$ is a uniform edge among all those missing from
the induced subgraph on $\cB_k$ at the end of the previous iteration.

Applying Lemma~\ref{lem-supercritical} to $G'=(V',E')$, we find that
with probability $1-o(n^{-3})$, every subset $X$ in $V'$ of size
$|X|\le n/5$ has at least $2|X|$ external neighbors in $G'$. We
claim that it is moreover connected with the same probability:
Indeed, the only edges in $E(\Gk)\setminus E(G')$ are edges in
$E(G)\setminus E(H)$ with both endpoints in $\cB_k$, yet by
assumption $\cB_k$ is connected already in the subgraph $H\subset
G'$ since $\Gk[H]$ is guaranteed to contain a Hamilton cycle.
Therefore, the fact that $\Gk$ is connected carries to $G'$.

We are now in a position to conclude the proof in the same manner used to prove Theorem~\ref{thm-ham-small-c}.
Considering all points in $V'\setminus \cB_k$ as trivial cycles (of which there are strictly less than $0.001 n$ since $|\cB_k|>0.999 n$) and adding those to the simple cycle that goes through all vertices of $\cB_k$ in $\Gk[H]$ we arrive at a 2-factor in $G'$ with at most $0.001 n+1$ cycles.
By the expansion properties of $G'$ that were detailed above, we can derive from Lemma~\ref{lem-Posa} that there are at least $n^2/50$ boosters in $G'$.
However, all but $0.001 n^2$ such boosters have both of their endpoints in $\cB_k$, therefore the probability that a new edge, uniformly chosen out of all edges missing from the induced subgraph on $\cB_k$, is a booster, is at least $1/50$ (with room to spare).
Seeing as $G'$ is connected, sprinkling any such edge to $G'$ would increase the length of the longest path in $G'$, and encountering $0.001 n+1$ boosters would culminate in a Hamilton cycle through the vertices of $V'$. Iterating this procedure for $J \geq n/10$ rounds, the number of boosters encountered before Hamiltonicity is achieved stochastically dominates a binomial random variable $Z \sim \bin(n/10, 1/50)$. Since $Z$ concentrates around its mean of $0.002 n$ we deduce that a Hamilton cycle will be formed except with probability exponentially small in $n$, as desired.
\end{proof}

\section{Expansion and factors in cores}\label{sec:core-properties}
This section contains the proofs of Lemmas~\ref{lem-critical}
and~\ref{lem-supercritical}, which address properties of the random
graph near and beyond the $k$-core threshold. These two regimes are
studied in Sections~\ref{subsec:gnp-near-tau} and~\ref{subsec:gnp-beyond-tau},
respectively, yet we begin with a
straightforward fact on the expansion of small sets, to be used in
both regimes.

\begin{claim}\label{clm-G'-expansion}
Fix $k$, let $G\sim\cG(n,p)$ for $p=\lambda/n$ with $1<\lambda<\log^2 n$. With probability $1-o(n^{-3})$ the following holds for any subgraph $G'=(V',E')\subset G $ (not necessarily induced) with $\delta(G')\geq k$:
\begin{compactenum}[(i)]
  \item If $k\geq 15$ then $|N_{G'}(X)\setminus X|>3 |X|$ for every $X\subset V'$ of size $|X| \leq \frac37 \lambda^{-15/7}n$.
  \item If $k\geq 15$ then $|N_{G'}(X)\setminus X|>2|X|$ for every $X\subset V'$ of size $|X| \leq \frac45 \lambda^{-5/3} n$.
  \item If $k\geq 14$ then $|N_{G'}(X)\setminus X|> 2 |X|$ for every $X\subset V'$ of size $|X| \leq \frac12 \lambda^{-7/4} n$.
\end{compactenum}
\end{claim}
\begin{proof}
Consider some $X\subset V'$ of size at most $\theta n$ that does not expand to at least $b|X|$ external neighbors in $G'$ for some fixed $b\geq2$, and let $Y\subset V'$ denote the
external neighbors of $X$ in $G'$. The fact that $\delta(G')\geq k$ implies that out of the $\binom{|X|}2+|X||Y|$ potential edges incident to $X$, at least $k|X|/2$ edges exist in
$G'$,
hence also in $G$. Whenever $b \leq k/2 -1 $, the probability of this event is thus at most
\begin{eqnarray}
 &&  \sum_{x \leq \theta n}\binom{n}{x}\binom{n}{b x}\binom{\frac{x^2}2 + b x^2}{kx/2} p^{kx/2}
  \leq \sum_{x \leq \theta n} \left[\frac{e n}x \Big(\frac{e n}{bx}\Big)^b \Big(\frac{(2b+1) e x}{k}\Big)^{\frac{k}2} \Big(\frac{\lambda}n\Big)^{\frac{k}2}\right]^x \nonumber\\
&\leq& \sum_{x=1}^{\lfloor25\log n\rfloor} \left[n^{-k/2+b+1+o(1)} \right]^x
 + \sum_{x=\lceil 25\log n \rceil}^{\lfloor \theta n \rfloor} \left[\frac{(e/\theta)^{b+1}}{b^b} \big((2b+1) e \lambda \theta/k \big)^{\frac{k}2} \right]^x\,.  \label{eq-b-expansion}
\end{eqnarray}
Taking $b=3$ and $\theta=\frac37\lambda^{-15/7}$, the first sum is $o(n^{-3})$, while for $k = 15$ the occurrences of $\lambda$ in the second summation cancel (and for $k>15$ the exponent of $\lambda$ will be negative and we recall that $\lambda>1$), thus reducing that summation into at most
\[ \sum_{x=\lceil 25\log n \rceil}^{\lfloor \frac37\lambda^{-15/7} n\rfloor}\left[\frac{(7e/3)^4}{27} \Big(\frac{3 e }{k }\Big)^{\frac{k}2}
\right]^x \leq \sum_{x=\lceil 25\log n \rceil}^{\lfloor \frac37\lambda^{-15/7} n\rfloor}e^{-x/3} =o( n^{-8})\,,
\]
where we used the fact that $\frac{(7e/3)^4}{27}(3e/k)^{k/2} \leq \frac{2}3$ for any $k\geq 15$.

Similarly, taking $b=2$ and $\theta=\frac45 \lambda^{-5/3}$, the first summation in~\eqref{eq-b-expansion} remains $o(n^{-3})$ and again the choice of $\theta$ allows us to omit all occurrences of $\lambda$ in the second summation for $k\geq 15$ (as $\lambda>1$ and its exponent is non-positive), reducing it into at most
\[ \sum_{x=\lceil 25\log n \rceil}^{\lfloor \frac45 \lambda^{-5/3} n\rfloor}\left[\frac{(5e)^3}{256} \Big(\frac{4 e }{k }\Big)^{\frac{k}2}
\right]^x \leq \sum_{x=\lceil 25\log n \rceil}^{\lfloor \frac45 \lambda^{-5/3} n\rfloor}e^{-x/8} = o(n^{-3})\,,
\]
yielding the second statement of the claim.
Finally, when $k\geq 14$, taking $b = 2$ and $\theta = \frac12\lambda^{-7/4}$ maintains the first summation in~\eqref{eq-b-expansion} at $o(n^{-3})$, and the second summation becomes at most
\[ \sum_{x=\lceil 25\log n \rceil}^{\lfloor \frac12 \lambda^{-7/4} n\rfloor}\left[\frac{(2e)^3}{4} \Big(\frac{5 e }{2k }\Big)^{\frac{k}2}
\right]^x \leq \sum_{x=\lceil 25\log n \rceil}^{\lfloor \frac12 \lambda^{-7/4} n\rfloor}e^{-x} = o(n^{-3})\,,
\]
as required.
\end{proof}

\subsection{Properties of sparse random graphs near the core threshold}\label{subsec:gnp-near-tau}
Throughout this subsection we will restrict our attention to $\cG(n,p)$ in the regime $p=O(1/n)$.
The next claim establishes that in any induced subgraph of minimum degree $k-1$, sets that are large enough --- namely, ones whose size is at least $\sqrt{n}$ --- have many external neighbors.
This can be established up to a linear scale, e.g., for sets up to size $n/15$, solely based on an assumption that the minimum degree of the graph under consideration (a subgraph of our random graph) is at least $k$. However, as we will later see, it will be imperative to obtain this estimate for larger sets, to which end we will rely on additional properties that are available to us in the framework of Lemma~\ref{lem-critical}.
(Compare the set sizes handled below, up to $\frac27 n$, with those in Claim~\ref{clm-G'-expansion}, up to $\frac45c^{-5/3}n$ in $\cG(n,c/n)$.)

\begin{claim}\label{clm-G'-expansion-large-sets}
Fix $k\geq 15$ and let $G\sim\cG(n,c/n)$ for some fixed $ k\leq c \leq \frac73 k$.
The following holds with probability at least $1-O(0.8^{\sqrt{n}})$. If $G'=(V',E')$ is an induced subgraph of $G$ on $|V'|\geq \frac45n$ vertices and $H=(V',E'\setminus (E_1\cup E_2))$ where $E_1$ is a matching and $|E_2|\leq n/\log\log n$
such that the minimum degree of $H$ is $\delta(H)\geq k-1$, then
 every $X\subset V'$ of size $\sqrt{n} \leq |X| \leq \frac27 n$ satisfies $|N_H(X)\setminus X| > n^{1/3}$.
\end{claim}

\begin{proof}
First consider the range $\sqrt{n}\leq |X|\leq \theta_1 n$ for $\theta_1=\frac2{23}$. Exactly as in the proof of Claim~\ref{clm-G'-expansion}, the probability that there exists a subgraph $H$ of $G$ with minimum degree $k-1$ (not necessarily formed by deleting the edge sets $E_1$ and $E_2$ as above) within which we can find a set $X\subset V$ of size $x$ such that $\sqrt{n}\leq x\leq \theta_1 n$ and all of its external neighbors belong to some $Y\subset V$ of size $y=n^{1/3}$ is at most
\begin{equation}
     \label{eq-medium-set-expansion}
 \sum_{\sqrt{n}\leq x \leq \theta_1 n} \left[\frac{e n}x \Big(\frac{e n}{y}\Big)^{y/x} \Big(\frac{e (x+2y)}{k-1}\Big)^{\frac{k-1}2} \Big(\frac{c}n\Big)^{\frac{k-1}2}\right]^x \,,
\end{equation}
where in comparison with~\eqref{eq-b-expansion} here we let $b$ assume the role of $y/x$.
Examining the base of the exponent in each summand we see that $(en/y)^{y/x} \leq \exp(n^{-1/6+o(1)})=1+o(1)$ and similarly the other appearance of $y$ are asymptotically negligible. Since $k\leq c\leq \frac73k \leq \frac52(k-1)$ the mentioned expression is at most
\[ (1+o(1)) \frac{en}x \left(\frac{5 e x}{2 n}\right)^{\frac{k-1}2} \leq (e/\theta_1+o(1)) \left(5e\theta_1/2\right)^{\frac{k-1}2} < 0.8 \,,\]
where the first inequality used the fact that $x \leq \theta_1 n$ and the second one holds for any $k\geq 15$. The occurrence of a bad set $X$ in this range of $x$ therefore has an overall probability of $O(0.8^{\sqrt{n}})$.

Now take $\theta_2 = \frac27$. To treat subsets $X$ of size $\theta_1 n \leq x \leq \theta_2 n$ we will take advantage of the additional hypothesis about the edge set of $H$. For a given choice of $V',X,Y$ such that all external neighbors of $X$ are confined to $Y$ (whose size we recall is $y=n^{1/3}$), there can be at most $|X|+|E_2|\leq|X|+n/\log\log n$ edges between $X$ and $V'\setminus (X\cup Y)$ in $G$. The variable $\Gamma$ counting the number of such edges has a law of $\bin(x(n'-x-y),p)$, and in particular its mean is $x(n'-x-y)p \geq x(\frac45 n- \frac27 n-o(n))p > (c/2)x \geq 7x$ for large $n$, since $c \geq k \geq 15$. Thus, writing $\delta_n = 12/\log\log n$ so that $x+n/\log\log n \leq (1+\delta_n)x$,
\begin{align*}
\P\left(\Gamma \leq (1+\delta_n)x \right) &\leq (1+\delta_n)x \binom{x(n'-x-y)}{(1+\delta_n)x} p^{(1+\delta_n)x} e^{-px(n'-x-y)+p(1+\delta_n)x} \\
&\leq  \left[(1+o(1))e c(1-x/n) e^{-c (n'-x)/n} \right]^x \leq  \left[(1+o(1))c e^{1-c/2} \right]^x\,,
  \end{align*}
  where the last inequality plugged in the facts $n'\geq \frac45 n$ and $x \leq \frac27 n$.
  Armed with this expression we revisit~\eqref{eq-medium-set-expansion} and revise it to incorporate the above probability, albeit at the cost of accounting the $\binom{n}{n'}$ possible choices for $V'$, so the probability of encountering the mentioned set $X$ becomes at most
\begin{align*}
&\sum_{\theta_1 n\leq x \leq \theta_2 n}\ \sum_{\frac45 n \leq n' \leq n} \binom{n}{n'}\left[(1+o(1))\frac{e n}x \Big(\frac{e c x }{(k-1)n}\Big)^{\frac{k-1}2} c e^{1-c/2}\right]^x \\
&\leq \sum_{\theta_1 n\leq x \leq \theta_2 n}\left[(1+o(1))e^{2+ h(\frac15)\frac{n}x} \ \frac{n}x \Big(\frac{e c x }{(k-1)n}\Big)^{\frac{k-1}2} c e^{-c/2}\right]^x\,,
  \end{align*}
  where we used the fact that $\sum_{i\leq \alpha n}\binom{n}{i} \leq \exp(h(\alpha)n)$ with $h(x)=-x\log x - (1-x)\log(1-x)$ being the entropy function, and
  as argued before, the terms involving $y$ in~\eqref{eq-medium-set-expansion} are easily absorbed in the $(1+o(1))$-factor.
Writing $\tilde{c}=c/(k-1)$ and $\tilde{x} = x/n$ (so that $\tilde{c} \leq \frac52$ and $\theta_1 \leq \tilde{x} \leq \theta_2$) we can rearrange the base of this exponent in each summand and find that it is asymptotically
equal to
\begin{align*}
  e^2 \left(\tilde{x} e^{\frac{2}{k-3} h(\frac15) /\tilde{x}}   \right)^{\frac{k-3}{2}} \left(\tilde{c} e^{1-\tilde{c}} \ c^{\frac{2}{k-1}}\right)^{\frac{k-1}2}
  \leq e^2 \left(\tilde{x} e^{\frac{1}{6} h(\frac15) /\tilde{x}}   \right)^{\frac{k-3}{2}} \left(\frac53 \right)^{\frac{k-1}2}
\end{align*}
using the facts that $k\geq 15$, that $c^{\frac2{k-1}}\leq (\frac52(k-1))^{\frac2{k-1}} \leq 5/3$ for $k\geq 15$ and that the function $t\mapsto t e^{-t}$ with $t>0$ has a global maximum at $t=1$. Similarly, the function $t\mapsto t e^{\alpha / t}$ with $\alpha,t>0$ is increasing for $t\geq \alpha$ and as $\tilde{x} \geq \theta = \frac2{23} > \frac16 h(\frac15)$ this implies that the above expression is maximized at $\tilde{x}=\frac27$, in which case it evaluates into
\begin{align*}
 \frac53 e^2 \left(\frac{10}{21} e^{\frac{7}{12} h(\frac15) }   \right)^{\frac{k-3}{2}}
< 0.9
\end{align*}
for any $k\geq 15$, and so the probability of encountering the mentioned set $X$ in this range is $O(\exp(-an))$ for some absolute constant $a>0$, which completes the proof.
\end{proof}

In what follows, for a graph $H$ and a subset $A$ of its vertices let $\odd(H\setminus A)$ denote the number of components of odd size in the subgraph obtained by deleting the vertices of $A$ from $H$.
The above claims are already enough to imply a bound on $\odd(H\setminus A)$ for small subsets $A$, which will later used as part of the proof of Lemma~\ref{lem-critical}.

\begin{corollary}\label{cor-tutte-small-sets}
Fix $k\geq 15$ and let $G\sim\cG(n,c/n)$ for some fixed $c$ such that $\max\{k, 20\} \leq c \leq \frac73 k$.
Then with probability $1-o(n^{-3})$, for any induced subgraph $G'=(V',E') \subset G$ of minimum degree at least $k$
on $|V'|\geq \frac45n$ vertices and any $H=(V',E'\setminus (E_1\cup E_2))$ such that $E_1$ is a matching, $|E_2|\leq n/\log\log n$ and $\delta(H)\geq k-1$, we
have $\odd(H\setminus X)\leq |X|$ for all $X\subset V'$ of size $|X|\leq \frac37 c^{-15/7} n$.
\end{corollary}
\begin{proof}
Let $\theta=\frac37 c^{-15/7}$ and $\eta = \frac12 c^{-7/4}$, and assume that the events stated in Claims~\ref{clm-G'-expansion} and~\ref{clm-G'-expansion-large-sets} hold. In particular,
every set $S\subset V'$ in the given graph $H$ satisfies:
 \begin{compactitem}
   \item If $|S| \leq \eta n$ then $|N_H(S)\setminus S| \geq 2|S|$ (by Part~(iii) of Claim~\ref{clm-G'-expansion} applied to $H$).
   \item If $\eta n \leq |S|\leq\theta n$ then $|N_H(S)\setminus S| \geq (2-o(1))|S|$ (by Part~(i) of Claim~\ref{clm-G'-expansion} applied to $G'$; moving to $H$ costs at most $|S|$ vertices due to the matching edges of $E_1$ plus $o(n)$ vertices due to $E_2$).
   \item If $\sqrt{n}\leq |S| \leq \frac27 n$ then $|N_H(S)\setminus S|\geq n^{1/3}$ (by Claim~\ref{clm-G'-expansion-large-sets} applied to $H$).
 \end{compactitem}

Suppose first that $\log n \leq |X|\leq (\theta/2)n$ and let $\{C_1,\ldots,C_s\}$ be all the connected components in $H\setminus X$ of size $|C_i|\leq \theta n$. Clearly $(N_H(C_i)\setminus C_i)\subset X$, but at the same time $|N_{H}(C_i)\setminus C_i| > (2-o(1))|C_i|$ hence we must have $|C_i| < (\frac12+o(1))|X|$ for all $i$. Suppose that $D$ is a minimal union of $C_i$'s such that $|D| > |X|-\frac13\log n$. Since $|X|\geq \log n$ we have  $|X|< \frac32 |D| $. Using that $|C_i| < (\frac12+o(1))|X|$, by the minimality of $D$, we also have $|D| \leq |X|-\frac13\log n +(\frac12+o(1))|X|< (\frac32+o(1))|X|<\theta n$ for large enough $n$. Therefore $|N_H(D)\setminus D|\geq (2-o(1))|D|$. However, at the same time $|N_H(D)\setminus D|\leq |X| < \frac32 |D|$ as there are no edges between distinct $C_i,C_j$, contradiction (for $n$, and hence $D$, large enough). Consequently, there cannot be more than $|X|-\frac13\log n$ components of size at most $\theta n$ (in fact, we showed that the union of all such components has cardinality at most $|X|-\frac13\log n$). When added to at most $1/\theta$ components of $H\setminus X$ whose size exceeds $\theta n$ we readily reach the required inequality $\odd(H\setminus X) \leq |X|$ in this case.

Next, consider the case $(\theta/2)n \leq |X| \leq \theta n$. Let us repeat the above argument, this time only collecting components of size at most $\log n$.
Taking $D$ to be a union of such components so that $\frac23|X|-\log n < |D| \leq \frac23|X|$ we see that $|N_H(D)\setminus D | \leq |X| \leq \frac32|D|+\frac32\log n$, which for $n$ large enough contradicts the expansion assumption $|N_H(D)\setminus D|\geq (2-o(1))|D|$. We deduce that there are at most $\frac23|X|$ such components, and adding at most $n/\log n = o(|X|)$ components of size larger than $\log n$ gives $\odd(H\setminus X)\leq |X|$.

If $3\leq |X| \leq \log n$ we observe that by our assumption there are no connected components of $H\setminus X$ of sizes between $\sqrt{n}$ and $\frac27 n$.
Consider the components $C_1,\ldots,C_s$ whose sizes are at most $\sqrt{n}$. Since there can be at most $3$ components of size larger than $\frac27 n$, it follows that if $|\cup_i C_i| \leq |X|-3$ then accounting for all such components results in $\odd(H\setminus X)\leq |X|$, as required. Otherwise, there exists a minimal union $D$ of $C_{i}$'s such that $|D| \geq |X|-2$. Being well in the range of the expansion hypothesis (this union has size at most $\sqrt{n}\log n$), the fact that $|X| \geq |N_H(D)\setminus D| \geq 2|D|$ implies that $|X| \leq 4$. However, in this case necessarily there are no small components in $H\setminus X$, since otherwise each $C_i$ would have to be size at most $|X|/2 $ by the expansion property, thus a vertex $v\in C_i$ would have degree at most $|C_i|-1+|X| \leq \frac32|X|-1 \leq 5$ (recall that $\delta(H)\geq k-1\geq 14$).

Finally, if $0\leq |X|\leq 2$ then no component of $H\setminus X$ can have size less than $\frac27 n$ (recall that such components had to satisfy $|C_i|<|X|/2$ and are thus empty by the discussion in the previous paragraph). At the same time, the probability that two subsets $A$ and $B$ of size $\frac27 n$ each have no edges between them in $G$ except for a matching (corresponding to $E_1$) is readily bounded by the probability that the degree into $B$ of every vertex in $A$ is at most $1$, translating to an upper bound of
\begin{align*}
\binom{n}{\frac47 n} \binom{\frac47n}{\frac27 n} \left( \left[\frac27 np +(1-p)\right] (1-p)^{\frac27 n -1}\right)^{\frac27 n} \leq
\bigg(e^{\frac72 h(\frac47)} 2^2 \left(\frac27 c + 1-o(1)\right) e^{-\frac27 c}\bigg)^{\frac27 n} \leq (0.99)^{2n/7}\,,
\end{align*}
where the last inequality holds for any $c \geq 20$. One can now repeat this calculation, this time taking into account the edges of $E_2$. Since their number ($n/\log\log n$) is negligible in comparison with the sizes of $A$ and $B$, they contribute a factor of at most $\binom{(2n/7)^2}{n/\log\log n}p^{n/\log\log n}=e^{o(n)}$ in the above estimate and hence the conclusion remains valid. This concludes the proof.
\end{proof}

A prerequisite for treating sets of larger size, which will also be useful to separately bound the size of the 2-factor in $G$ once we establish its existence, is the next estimate on the maximal number of vertex-disjoint cycles in $G$.

\begin{claim}\label{clm-max-disjoint-cycles}
Fix $c>0$ and let $G \sim\cG(n,c/n)$. Let $Y$ count the maximal number of vertex-disjoint nontrivial cycles in $G$. Then $\P\big( Y \geq \frac12 n/\sqrt{\log n} \big) < e^{-\sqrt{n}}$ for any sufficiently large $n$.
\end{claim}
\begin{proof}
Let $Y_0$ be the maximal number of vertex-disjoint cycles in $G$ such that each cycle is of length at most $L=4\sqrt{\log n}$. Since $Y \leq Y_0 +  n / L $ it suffices to show that $\P\big(Y_0 \geq n / L \big) \leq e^{-\sqrt{n}}$.
Observe that
\[ \E Y_0 \leq \sum_{s \leq L} \binom{n}{s} (s-1)! \left(\frac{c}n\right)^{s} \leq
\sum_{s \leq L} c^s = e^{O(L)} = n^{o(1)}<n/(2L)\,,
\]
and consider the vertex-exposure martingale for $Y_0$ in $\cG(n,c/n)$. Its Lipschitz constant is clearly equal to $1$ since the addition of a new vertex to a graph retains every existing subset of disjoint cycles whereas it can increase the cycle count by at most 1. Hence, Hoeffding's inequality gives
\[ \P\left(Y_0 \geq n/L \right) \leq \P\left(Y_0 - \E Y_0 \geq n/(2L)\right) \leq e^{-O(n/L^2)} < e^{-\sqrt{n}}\,,\]
where the last inequality is valid for large enough $n$.
\end{proof}
The final ingredient needed for proving Lemma~\ref{lem-critical} is the following claim, which verifies Tutte's condition for large sets and complements the range of sets that were covered by Corollary~\ref{cor-tutte-small-sets}.
\begin{claim}\label{clm-tutte-large-sets}
Fix $k\geq 15$, let $G\sim\cG(n,c/n)$ for $k \leq c \leq \frac73 k$ and let $G'=(V',E')$ be an induced subgraph of $G$ with $\delta(G')\geq k$.
Let $H=(V',E'\setminus(E_1\cup E_2))$ where $E_1$ is a matching and $|E_2| \leq n/\log\log n$.
Then with probability $1-O(e^{-n/40})$ we have
$\odd(H \setminus X) \leq |X|$ for all $X \subset V'$ such that $ |X| \geq \frac37 c^{-15/7} n$.
\end{claim}
\begin{proof}
Set $x=|X|$ and consider first subsets of size
\begin{equation}
  \label{eq-X-size-1}
  \frac{3n}{7c^{15/7}}\leq x \leq \frac{n}c\,.
\end{equation}
Observe that by Claim~\ref{clm-max-disjoint-cycles} with probability at least $1-O(e^{-\sqrt{n}})$ the maximal number of vertex disjoint cycles in $G$, and hence also in $H$, is at most $n/\sqrt{\log n}$. In particular, there can be at most that many connected components of $H\setminus X$ containing cycles. Disregard these components, as well as at most $n/\log\log n$ additional components incident to the edge set $E_2$.
Altogether we discarded $o(n)$ components, hence if $\odd(H\setminus X) > |X|$ for a linear $X$ as above then we will
be left with at least $\frac{15}{16} x$ odd components in $H\setminus X$ for large enough $n$.
Let $Y$ be a set comprised of an arbitrary leaf from each of these components (which are trees by construction).
The vertices of $Y$ have degree at least $k$ in $G'$ and as such they have at least $k-2$ neighbors in $X$ in the graph $H$ (plus a single neighbor in their tree component of $H\setminus X$ and possibly one additional neighbor in $E_1$).
The probability that there exists such a pair $(X,Y)$ in $G$ is at most
\[
\sum_{x=\frac37 c^{-\frac{15}7}n}^{ n/c} \!\!\binom{n}x \binom{n}{\frac{15}{16}x} \Big[\P(\bin(x,c/n) \geq k-2)\Big]^{\frac{15}{16} x}
\leq \!\!
\sum_{x= \frac37 c^{-\frac{15}7}n}^{n/c} \!\!\bigg[\frac{16}{15} \left(\frac{en}x\right)^{\frac{31}{15}} \P(\bin(x,c/n) \geq k-2)\bigg]^{\frac{15}{16} x}.
\]
Plugging in the well-known estimate (see~\cites{Bennett,Hoeffding}) that if $Z$ is a binomial variable and $a >0$ then $\P\left(Z - \E Z \geq a\right) \leq \exp\left( -\phi\big(a/\E Z\big) \E Z\right)$ for $\phi(x) = (1+x)\log(1+x)-x$, we get that
\begin{align*}
\frac{16}{15} \left(\frac73 ec^{15/7}\right)^{\frac{31}{15}} \P(\bin(x,c/n) \geq k-2) &\leq \frac{16}{15} \left(\frac73 e (7k/3)^{15/7}\right)^{\frac{31}{15}}  (k-2)^{-(k-2)} e^{k-3} \\
\leq 2 e^6 \, k^{5} \left(\frac{e}{k-2}\right)^{k-2}&=2e^{11} \left(1+\frac{2}{k-2}\right)^5 \left(\frac{e}{k-2}\right)^{k-7} < 0.95\,,
\end{align*}
where the first inequality uses the fact that $c \leq \frac73 k$ and $x c / n \leq 1$, while the last inequality holds for $k\geq 15$. Combined with the previous inequality this implies that the probability of encountering some $X$ violating Tutte's condition in $H$ and with size as in~\eqref{eq-X-size-1}
is at most $\exp(-c' n)$ for some fixed $c'>0$.

We will now handle the range
\begin{equation}
  \label{eq-X-size-2}
  \frac{n}{c}\leq x \leq (k-3)\frac{n}{c}\,.
\end{equation}
Given a candidate subset $X$ of size $x$ for which $\odd(H\setminus X) > x$ we again consider the above defined set $Y$ (comprised of an arbitrary leaf from each tree component of $H\setminus X$ not incident to any of the edges of $E_2$).
 Previously, we only used the fact that each $v\in Y$ has at least $k-2$ neighbors in $X$.
Now we will use the fact that one can take $Y$ to have size $(1-o(1))x$, and in addition $Y$ is an independent set in $H$ (its vertices fall in distinct components of $H\setminus X$). The probability that a given subset $Y$ of size $y$ in $G$ spans precisely $t$ edges is at most
\[ \frac{\binom{y}2^t}{t!}p^t(1-p)^{\binom{y}2-t} \leq \frac1{t!}\left(\frac{\binom{y}2 p }{1-p}\right)^{t} (1-p)^{\binom{y}2 } =: \zeta_t\,,\]
which is maximized at $t=\binom{y}2 \frac{p}{1-p} \geq (1-o(1))y/2$, with the last inequality due to $x\geq n/c$.
Of course, the values of $t$ we are considering cannot exceed $y/2$ as the edges of $G$ in $Y$ are only allowed to form a matching (recall that $Y$ is not incident to $E_2$),
thus by Stirling's formula
\[ \sum_{t\leq y/2}\zeta_t \leq \bigg(\frac{y}2+1\bigg) \frac{1+o(1)}{\sqrt{\pi y}}\bigg(\frac{\binom{y}2 p }{1-p}\,\frac{2e+o(1)}{y} (1-p)^{(y-1)}\bigg)^{(\frac12-o(1))y}
\leq \left((e+o(1)) p x e^{-p x}\right)^{(\frac12-o(1))x} \,.\]
Combining this estimate with the number of choices for the sets $X, Y$ and with the probability that each vertex of $Y$ has at least $k-2$ neighbors in $X$, we arrive at the following upper bound on the event of encountering $X,Y$ as above:
\begin{align}
\sum_{x= n/c}^{(k-3) n/c}
\left[ (1+o(1))\left(\frac{e n}{x}\right)^{2} \P(\bin(x,c/n) \geq k-2) \sqrt{ e p x e^{-p x}}\right]^{x}\,,\label{eq-B-independent-XY}
\end{align}
where within each summand we could safely absorb all the $o(1)$-terms in the exponents into the $o(1)$-term of the leading constant since the base of each of these is uniformly bounded.

Notice that for $\ell\geq k-2$ we have
\[ \frac{\P(\bin(x,p) = \ell+1)}{\P(\bin(x,p) = \ell)} = \frac{x-\ell}{\ell+1} \,\frac{p}{1-p} \leq (1+o(1))\frac{px}{k-1} \,,\]
which by our hypothesis on $x$ is at most $(1+o(1)) (k-3)/(k-1) < 1$ for large enough $n$. It follows that
\[ \P(\bin(x,c/n) \geq k-2) \leq \frac{1+o(1)}{1- \frac{px}{k-1}} \P(\bin(x,c/n)=k-2)\leq \frac{1+o(1)}{1- \frac{px}{k-1}} \frac{(px)^{k-2}}{(k-2)!} e^{-px} \,,\]
and plugging this in~\eqref{eq-B-independent-XY} gives an upper bound of
\[ \sum_{x=n/c}^{(k-3) n/c}
\left[ (1+o(1))\left(\frac{e n}{x}\right)^{2} \frac{1}{1- \frac{px}{k-1}} \frac{(px)^{k-2}}{(k-2)!} e^{-px}\sqrt{ e p x e^{-p x}}\right]^{x}\,.
\]
Writing $z = px$ and $\tilde{c}=c/k$, we will denote the (asymptotic) base of the exponent above by
\[ f_k(z) = \tilde{c}^2 e^{-\frac32 z + \frac52} z^{k-\frac72}\frac{k^2 (k-1)}{(k-1-z)(k-2)!}\,.\]
With this notation we aim to bound $f_k(z)$ away from $1$ for all $1 \leq z \leq k-3$ and $1 \leq \tilde{c} \leq \frac73$.
It is easy to verify that the local extrema of $f_k$ are precisely the roots $z^*_k < z^{**}_k$ of the quadratic
$ 3z^2 + (12-5k)z + (2k^2-9k+7)$, where $z^*_k$ is a local maximum of $f_k$ and $z^{**}_k$ is a local minimum. Furthermore,
since
\[ z^{**}_k = \frac56 k - 2 + \frac16\sqrt{k^2-12k+60} > k - 3\]
(the last inequality being valid for any $k>0$) we can conclude that throughout the
range $1\leq z\leq k-3$ the function $f_k(z)$ is maximized
at
\[ z^*_k = \frac56 k - 2 - \frac16\sqrt{k^2-12k+60}\,.\]
Observe that $z^*_{k+1} \geq z^*_k$ and
\[ \frac23 k - \frac43 \leq z^*_k \leq \frac23 k - 1\,. \]
In particular $z^*_{k+1} \leq \frac23k - \frac13 $ and $0 \leq z^*_{k+1}-z^*_k \leq 1$.

Using the bound $z^*_k \leq \frac23k - 1$ in the expression $1/(k-1-z)$ we immediately get that $f_k(z^*_k) \leq g_k(z^*_k)$ where
\[ g_k(z) = 3 \tilde{c}^2 e^{-\frac32 z + \frac52} z^{k-\frac72} \frac{k (k-1)}{(k-2)!}\,.\]

For ease of notation, let $z_0 = z^*_k$ and $z_1 = z^*_{k+1}$. We have
\[ \frac{g_{k+1}(z_1)}{g_{k}(z_0)} = e^{-\frac32(z_1-z_0)} \left( \frac{z_1}{z_0}\right)^{k-\frac72} z_1 \,\frac{k+1}{(k-1)^2}\,. \]
Observe that
\[ \left(\frac{z_1}{z_0}\right)^{k-\frac72} \leq e^{\frac{z_1-z_0}{z_0}(k-\frac72)} \leq
e^{\frac32 \frac{k-\frac72}{k - 2}(z_1-z_0)} \leq e^{\frac32(z_1-z_0)}\,,
\]
and therefore
\[ \frac{g_{k+1}(z_1)}{g_{k}(z_0)} \leq  z_1 \,\frac{k+1}{(k-1)^2} \leq \frac{\frac23(k-\frac12)}{k-1}\, \frac{k+1}{k-1} \leq \frac23 \left(1+\frac3{k-1}\right) < \frac56\,,\]
where the last inequality is valid for any $k\geq 15$. It therefore suffices to show that $g_k(z^*_k)$ is bounded away from $1$ for $k=15$.
Indeed, for $k=15$ we have
\[ z^*_{15} \in (10 - \tfrac54 , 10 - \tfrac{6}5)\]
and substituting $\tilde{c}$ by its maximal value of $\frac73$ gives
\[ g_{15}(z^*_{15}) \leq 49 e^{-\frac32 \left(10-\tfrac54\right) + \frac52}\, \left(10-\tfrac65\right)^{15-\frac72} \frac{70}{13!} < \frac{39}{40} < 1\,.\]
This establishes the desired statement for the range~\eqref{eq-X-size-2} provided that $1 \leq c/k \leq \frac73$, that is, $\odd(H\setminus X) \leq |X|$ for all sets $X$ whose sizes are between $n/c$ and $(k-3)n/c$ except with probability exponentially small in $n$.

To conclude the proof it remains to handle $x \geq (k-3)n/c$. Observe that if $c \leq 2(k-3)$ then the range~\eqref{eq-X-size-2} goes up to $x\leq n/2$, beyond which $\odd(H\setminus X)\leq |X|$ trivially holds. We are thus left with the case $2(k-3) \leq c \leq \frac73 k$ and $(k-3)n/c \leq x \leq n/2$, where we claim that w.h.p.\ there will not exist a set $Y$ as above which is an independent set in $H$. Indeed, the probability that $\cG(n,c/n)$ will contain a set $Y$ of size $(1-o(1))(k-3)n/c \leq y \leq n/2$ where the edges form a matching (recall that $Y$ is not incident with $E_2$) of size $\ell \leq y/2$ is at most
\begin{align*}
\binom{n}y\sum_{\ell\leq y/2}\binom{y}{2\ell}\frac{(2\ell)!}{\ell! 2^\ell} p^{\ell}(1-p)^{\binom{y}2 - \ell/2} &\leq \sum_{\ell\leq y/2}\bigg[ (e+o(1))\frac{n}{y} \left(\frac{e y}{2\ell}\right)^{2\ell/y}\left(\frac{2\ell c}{e n}\right)^{\ell/y} e^{-py / 2}\bigg]^y \\
&\leq \sum_{\ell\leq y/2}\bigg[ \frac{e c+o(1)}{k-3} \left(\frac{ey c}{4\ell }\right)^{\ell/y} e^{-(k-3)/ 2}\bigg]^y \,.
 \end{align*}
It is easy to verify that the function $x\mapsto (a x)^{1/x}$ is decreasing in $(e/a,\infty)$ for any $a >0$. In particular, for $a = e c / 4 > e/2 $ the maximum of $(a y/\ell)^{\ell/y} $ over all $y/\ell \geq 2$ is at most $\sqrt{2 a}$, hence each of the summands in the right-hand-side above is at most
\[ \bigg[ \frac{e c+o(1)}{k-3} \sqrt{\frac{ec}2} e^{-(k-3)/ 2}\bigg]^y
\leq \bigg[ \frac{c^{3/2} e^3 +o(1)}{\sqrt{2}(k-3)}  e^{-k/ 2}\bigg]^y
\leq \bigg[ 5 k^{3/2} e^{-k/2}\bigg]^{y} < \left(\frac1{5}\right)^{-y}
\,.\]
where the second inequality uses that $c \leq \frac73 k$ and that $(\frac73)^{3/2} \frac{e^3}{\sqrt{2}(k-3)} < 5$ for $k\geq 15$.
\end{proof}
We are now in a position to prove the main result of this subsection.
\begin{proof}[\textbf{\emph{Proof of Lemma~\ref{lem-critical}}}]
The fact that $G'$ is connected (as stated in Item~\eqref{it-lem-1} of the lemma)
follows from Corollary~\ref{cor-tutte-small-sets} for a choice of $X=\emptyset$.
Similarly, the expansion of sets of size at most $n/c^3$ (as stated in Item~\eqref{it-lem-2} of the lemma) is an
immediate consequence of Claim~\ref{clm-G'-expansion} (with plenty of room to spare, as $\frac54 c^{5/3} \leq c^2$ already for $c\geq 2$).
It remains to establish the conclusion on
the existence of a factor consisting of a path $P$ and at most $n/\sqrt{\log n}$ cycles in $G'$.

The combination of Corollary~\ref{cor-tutte-small-sets} and Claim~\ref{clm-tutte-large-sets} shows that existence of a near-perfect matching (one that misses at most one vertex) in $G'$ with probability at least $1-o(n^{-3})$. Indeed, in the case where $|V'|$ is even, the existence of a perfect matching in $G'$ follows from the
celebrated Tutte condition~\cite{Tutte} since $\odd(H\setminus A)\leq |A|$ for every $A\subset V'$. When $|V'|$ is odd, satisfying Tutte's condition
for every $A\neq\emptyset$ will imply a near-perfect matching via Berge's formula (see, e.g.,~\cite{LP}*{\S3.1.14}) since, as discussed above, $G'$ is connected with probability $1-o(n^{-3})$.
Denoting this matching by $M_1$ we then re-apply Corollary~\ref{cor-tutte-small-sets} and Claim~\ref{clm-tutte-large-sets} on $H=(V',E'\setminus M_1)$ to extract another near-perfect matching $M_2$ with the same success
probability. The factor formed by $M_1 \cup M_2$ is either a union of cycles (possibly with one isolated vertex which we count as a trivial cycle) or a union of cycles in addition to a single path.
Moreover, the number of cycles is at most $n/\sqrt{\log n}$ except with exponentially small probability thanks to
Claim~\ref{clm-max-disjoint-cycles}, as required.
\end{proof}

\subsection{Properties of sparse random graphs beyond the core threshold}\label{subsec:gnp-beyond-tau}

The main element needed for the proof of Lemma~\ref{lem-supercritical} is the expansion properties of sets in our random graph whose sizes exceed the one addressed in Claim~\ref{clm-G'-expansion}.

\begin{claim}\label{clm-expansion-large-c}
Let $G\sim\cG(n,p)$ for $p=\lambda/n$ with $32 \leq \lambda \leq \sqrt{n}$. There exists an absolute constant $a>0$ such that with probability at least $1-O(e^{-a n})$
we have $|N_G(X)\setminus X| > \frac52|X|$ for every $X\subset V(G)$ such that $\frac12 \lambda^{-5/3} n \leq |X| \leq n/5$.
\end{claim}
\begin{proof}
Let $\theta_1 = \theta_1(n) = \frac12 \lambda^{-5/3}$ and $\theta_2 = 1/5$. Consider a potential subset $X$ of $x$ vertices such that $\theta_1 n \leq x \leq \theta_2 n $ and yet
$|N_G(X)\setminus X| \leq b|X|$ for $b=\frac52$. Letting $\tilde{x} = x/n$ for brevity, the probability that there exists such a set $X$ is at most
\begin{eqnarray*}
\binom{n}{(b+1)x} \binom{(b+1)x}{x} (1-p)^{x(n-(b+1)x)} &\leq&
\left(\frac{en}{(b+1)x}\right)^{(b+1)x} \, e^{h(\frac1{b+1})(b+1)x}e^{-\frac{\lambda}{n}x(n-(b+1)x)}\\
&=& \left(\frac{e^{1+h(\frac1{b+1})-\frac{\lambda}{b+1}}}{b+1}
\frac{e^{\lambda \tilde{x}}}{\tilde{x}}  \right)^{(b+1)x}
\end{eqnarray*}
The unique extremum of $f(t) = e^{\lambda t}/t$ is a minimum at $t=1/\lambda$, thus for $\theta_1 \leq \tilde{x} \leq \theta_2$ we have
$f(\tilde{x}) \leq \max\{ f(\theta_1),\,f(\theta_2)\}$.
One can verify that
\[ f(\theta_1) = 2 \lambda^{5/3} e^{\lambda^{-2/3}/2} \leq 5 e^{\lambda/5} =  f(\theta_2) \]
for any $\lambda \geq 32$, and in particular the above probability is at most
  \[ \left( \frac{10}7 e^{1+h(2/7)-(2/7)\lambda + \lambda/5} \right)^{\frac72x} \leq
  \left( \frac{10}{7} e^{1+\log 2 - \lambda/12}  \right)^{\frac72x} < 2^{-2x}
  \,,\]
where in the last inequality is again valid for $\lambda \geq 32$. Summing over $O(n)$ values of $\theta_1 n \leq x \leq \theta_2 n$ now gives the desired result.
\end{proof}

Observe that the claim above analyzed expansion in the entire random graph, while we will be mostly interested in
the expansion properties of sets within the $k$-core. To this end we provide the following estimate on the size of core.

\begin{claim}\label{clm-core-size}
Fix $k\geq 15$, let $G\sim\cG(n,p)$ for $p = \lambda/n$ with $(7k-1)/3 \leq \lambda\leq\sqrt{n}$, and define
\begin{align}
  \label{eq-gamma-def}
  \gamma=\gamma(k,\lambda) = (1+\delta) e^k \left(\frac{\lambda}{k-1}\right)^{k-1} e^{-(1-\delta) \lambda}\quad\mbox{ for }\delta=10^{-3}\,.
\end{align}
Then $\gamma<10^{-3}$ and $|V(\Gk)|> (1-\gamma)n$ with probability at least $1-O(e^{-a n})$
for some absolute constant $a>0$.
\end{claim}
\begin{proof}
Consider a potential set $S$ of $s=\lfloor \gamma n\rfloor$ vertices that do not belong to the $k$-core.
We note that $\gamma$ is decreasing in $\lambda$ in the range $\lambda\geq k$ and so
$\gamma(k,\lambda) \leq \gamma(k,\lambda_0)$ for $\lambda_0(k) =(7k-1)/3 = \frac73(k-1) + 2$.
At this value of $\lambda$ we have
\begin{equation}
  \label{eq-gamma(k,lambda0)}
  \gamma(k,\lambda_0) = (1+\delta) e^{1-2(1-\delta)} \left(\frac{7e}3 \left(1 + \frac6{7(k-1)}\right)e^{-(1-\delta)\frac73}\right)^{k-1}\,,
\end{equation}
and as the last factor is less than $1$ for any $k\geq 4$, the entire expression is decreasing in $k$ and so it is upper bounded by its value at $k=15$, which,
as one can easily verify, is strictly less than $10^{-3}= \delta $.

The iterative construction of the $k$-core via successively deleting vertices of degree less than $k$ reveals that necessarily there are at most $(k-1)s$ edges in the cut between $S$ and $V(\Gk)$.
As the expected number of edges in this cut is $s(n-s)p \geq s \lambda/2 > s(k-1)$, the probability to encounter such a set $S$ is at most
\begin{align*}
\left(s(k-1)+1\right) \binom{n}{s}&\binom{s(n-s)}{s(k-1)} p^{s(k-1)}(1-p)^{s(n-s)-s(k-1)}\leq
\Bigg[(1+o(1))\frac{e n}{s e^{p(n-s)}} \left(\frac{e(n-s)p}{k-1} \right)^{k-1} \Bigg]^s \\
&=
\Bigg[\frac{e+o(1)}{\gamma e^{\lambda(1-\gamma)}}\left(\frac{e\lambda}{k-1}\right)^{k-1} \Bigg]^s
= \Bigg[\frac{1+o(1)}{ (1+\delta) \, e^{(\delta-\gamma)\lambda}}  \Bigg]^s < e^{- a s}
\end{align*}
for some absolute constant $a>0$, where the second equality substituted the definition of $\gamma$ in~\eqref{eq-gamma-def} and the last inequality justified by the fact that $\gamma < \delta$.
\end{proof}

\begin{proof}
  [\textbf{\emph{{Proof of Lemma~\ref{lem-supercritical}}}}]
Let $H\sim\cG(n,m)$ for $\frac16(7k - 1)n \leq m \leq n \log n$.
Note that the edge density in $H$ is $p\sim\lambda/n$ for $(7k-1)/3 \leq \lambda \leq 2\frac{\log n}n$,
and in particular for $k\geq 15$ we see that $\lambda$ satisfies the hypotheses of Claims~\ref{clm-expansion-large-c} and~\ref{clm-core-size}.
As per the conclusions of these claims, every set $X\subset V$ of size $\frac12 \lambda^{-5/3} n \leq |X| \leq n/5$ has $|N_H(X)\setminus X|>\frac52 |X|$, and $|V(\Gk[H])| \geq (1-\gamma)n$ for $\gamma < 0.001$ as defined
in~\eqref{eq-gamma-def}, except with probability $O(\exp(-a n))$ for some absolute constant $a>0$.
 Letting $G$ be the result of adding $n/8$ uniformly chosen new edges to $H$, its edge density is asymptotically $(\lambda+\frac14)/n$, and appealing to Claim~\ref{clm-G'-expansion} we find that with probability $1-o(n^{-3})$ every subgraph
$F\subset G$ with minimum degree $k$ has $|N_{F}(X)\setminus X| > 2|X|$
for any $X\subset V(F)$ of size at most $\frac45 \lambda^{-5/3} n$.
Condition on these properties of $H$ and $G$, from which we will now derive
the required expansion properties of $G'$.

By definition, the graph $G'$ has $V'=V(\Gk)$ as its vertex set, and
its minimum degree is at least $k$, since every vertex $u\in V' \setminus V(\Gk[H])$
has the same degree in $G'$ as it has in $\Gk$ (at least $k$),
whereas every $u\in V(\Gk[H])$
has degree at least $k$ already within $\Gk[H] \subset G'$. Thus,
thanks to our conditioning, every $X \subset V'$ such that $|X| \leq \frac45 \lambda^{-5/3} n$
satisfies $|N_{G'}(X)\setminus X|>2|X|$.

It remains to treat sets $X\subset V'$ of size $\frac45 \lambda^{-5/3} n \leq |X| \leq n/5$. We claim that the proof will be concluded once we show that
\begin{align}
  \label{eq-gamma-theta-relation}
  \gamma(k,\lambda) \leq \frac25 \lambda^{-5/3}\qquad \mbox{ for any }\lambda\geq (7k - 1)/3 \,.
\end{align}
To see this, recall by our assumptions, for any set $X$ in this size range, its set of external neighbors
$Y =N_H(X)\setminus X $ satisfies $|Y| > \frac52 |X|$. Out of this set $Y$, all but at most
$\gamma n $ vertices belong to $V'$ (and hence belong to $N_{G'}(X)\setminus X$, as the
induced subgraph of $H$ on $V'$ is a subgraph of $G'$). Thanks to~\eqref{eq-gamma-theta-relation}
we have $\gamma n \leq |X|/2$, thus in particular $|N_{G'}(X)\setminus X| > 2|X|$, as desired.

Turning our attention to~\eqref{eq-gamma-theta-relation} and recalling the definition of $\gamma$ from~\eqref{eq-gamma-def}, we need to show that
\[ (1+\delta) \left(\frac{e\lambda}{k-1}\right)^{k-1} e^{1-(1-\delta)\lambda} \cdot \frac52\lambda^{5/3} \leq 1\qquad\mbox{ for }\delta=10^{-3}\,.\]
Since the function $x\mapsto x^{k-1+5/3} \exp(-(1-\delta)x)$ has a unique extremum in the form of a global maximum at $x=(k-1+\frac53)/(1-\delta) < \frac73 k - 1$, it suffices to establish the above inequality at $\lambda_0 = \frac73 k -\frac13 = \frac73(k-1)+2$. Following~\eqref{eq-gamma(k,lambda0)}, the left-hand-side of the sought inequality
then becomes
\[
(1+\delta) \cdot 5\cdot 2^{2/3} \left(1+\frac76(k-1)\right)^{5/3} e^{1-2(1-\delta)} \left(\frac{7e}3 \left(1 + \frac6{7(k-1)}\right)e^{-(1-\delta)\frac73}\right)^{k-1}\,.
\]
Denoting this expression by $\xi_k$, we need to show that $\xi_k \leq 1$ for all $k\geq 15$. We see that for any $k\geq 2$,
\[ \frac{\xi_{k}}{\xi_{k-1}} = \left(1+\frac7{7k-8}\right)^{5/3} \left(\frac{7e}3 e^{-(1-\delta)\frac73} \right)
\left(1+\frac6{7(k-1)}\right)\left(1-\frac6{(k-1)(7k-8)}\right)^{k-2}\,.\]
Clearly each of the three factors that depend on $k$ is decreasing in $k$ (with the last of these being strictly less than $1$ and raised to a power growing with $k$). In particular,
for any $k\geq 5$ we have
\[\frac{\xi_k}{\xi_{k-1}} \leq \frac{\xi_5}{\xi_4} < 0.95 < 1\,.\]
Equivalently, $\xi_k$ is strictly decreasing in $k$ in the range $k\geq 5$, and so for any $k\geq 15$ we have
 \[ \xi_k < \xi_{15} < 0.95 < 1 \,,\]
thus establishing~\eqref{eq-gamma-theta-relation}, as required.
\end{proof}

\section{Packing edge-disjoint Hamilton cycles}\label{sec:ham-pack}

In this section we prove Theorem~\ref{thm-2}. Since several of the
ideas and techniques used here are rather similar to those applied
beforehand, we will allow ourselves to be somewhat brief at times.
Also, we will borrow extensively from the notation and terminology of the
previous sections.

First of all, let
\[
k_1=\left\lfloor\frac{k-3}{2}\right\rfloor \,.
\]

Similarly to Theorems \ref{thm-ham-small-c}
and~\ref{thm-ham-large-c}, the following two theorems cover the
critical regime (which is narrower here) and the supercritical
regime, respectively.

\begin{theorem}
  \label{thm-pack-small-c}
Let $k$ be a large enough constant, and let $G\sim \cG(n,m)$ for
$m=cn/2$, with $c$ satisfying $k\leq c \leq k+100\sqrt{k\log k}$.
Then with probability $1-o(n^{-2})$ either $\Gk$ is empty or it
contains a family of $k_1$ edge-disjoint Hamilton cycles.
\end{theorem}

\begin{theorem}
  \label{thm-pack-large-c}
Let $k$ be a large enough constant, and let $H\sim \cG(n,m)$ for
$m=cn/2$, with $c=c(n)$ satisfying $k+99\sqrt{k\log k}\le c\le 2\log
n$. Let $G$ be the graph obtained by adding $n$ uniformly
distributed new edges to $H$. Then the probability that $\Gk[H]$
contains a family of $k_1$ edge disjoint Hamilton cycles, and at the
same time $\Gk[G]$ does not contain a family of $k_1$ edge-disjoint
Hamilton cycles, is $o(n^{-2})$.
\end{theorem}

The above two theorems, combined with the classical result of
Bollob\'as and Frieze~\cite{BolF} that the random graph
$G\sim \cG(n,m)$ with $m=\frac{n}{2}\left(\log n +(k-1)\log\log n
+\omega_n\right)$ and any diverging increasing sequence $\omega_n$ w.h.p.\ contains
$\left\lfloor k/2\right\rfloor$ edge-disjoint Hamilton
cycles, clearly establishes Theorem~\ref{thm-2}.

In our proofs we prefer to work with the binomial random graph model
$\cG(n,p)$ instead of $\cG(n,m)$; standard arguments about the
essential equivalence of the models can be invoked to transfer our
claims to the model $\cG(n,m)$.

 We start  with the following lemma, guaranteeing typical local
 expansion in subgraphs of minimum degree at least 5 in sparse
 random graphs.

\begin{lemma}\label{lem-expansion}
Let $G\sim \cG(n,p)$ for $p=c/n$ with $c_0\le c(n)\le 2\log n$.
If $c_0$ is large enough, then with probability $1-o(n^{-2})$ every
subgraph $H$ of $G$ of minimum degree $\delta(H)\ge 5$ has the
following property: every subset $X\subset V(H)$ of size $|X|\le
n/c^{10}$ satisfies $|N_H(X)\setminus X|\ge 2|X|$.
\end{lemma}

\begin{proof}
First observe that, for any given constant $C_0>0$, with probability
$1-o(n^{-2})$ every set $X\subset V(G)$ with $|X|\le C_0$ vertices
contains at most $|X|+2$ edges. Indeed, the probability that there
exists a subset violating this claim is at most
\[
\sum_{i\le
C_0}\binom{n}{i}\binom{\binom{i}{2}}{i+3}p^{i+3}=O(n^{C_0}p^{C_0+3})=o(n^{-2})\,
.
\]
Now we argue that typically in $\cG(n,p)$ all larger sets are still
fairly sparse. Specifically, we claim that there exists a large
enough $C_0$ such that with probability $1-o(n^{-2})$ every set
$Z\subset V(G)$ of cardinality $C_0\le |Z|\le 3n/c^{10}$
spans at most $1.2|Z|$ edges in $G$. Indeed, the probability that there exists a set $Z$ of cardinality $|Z|=z$ violating this claim
is at most
\begin{eqnarray*}
\sum_{z=C_0}^{3n/c^{10}}\binom{n}{z}\binom{\binom{z}{2}}{1.2z}p^{1.2z}&\le&
\sum_{z=C_0}^{3n/c^{10}}\left[\frac{en}{z}\, \left(\frac{ezp}{2.4}\right)^{1.2}\right]^z \leq
\sum_{z=C_0}^{3n/c^{10}}\left[5\left(\frac{z}{n}\right)^{0.2}c^{1.2}\right]^z\\
&\leq& \sum_{z=C_0}^{\log n} n^{-0.1z}+\sum_{z=\log n}^{3n/c^{10}}\left[10c^{-0.8}\right]^z=o(n^{-2})\,,
\end{eqnarray*}
for $c\ge c_0$ large enough and $z\ge C_0$ large enough.

Now take $C_0\le |X|\le n/ c^{10}$ and let $Y=N_H(X)\setminus X$.
Since all degrees in $H$ are at least 5, we deduce that
 $e_H(X,Y)\ge 5|X|-2e_H(X)$, and thus $e_H(X\cup
 Y)=e_H(X)+e_H(X,Y)\ge 5|X|-e_H(X)\ge 3.8|X|$. If $|Y|< 2|X|$,
 then $|X\cup Y|< 3|X|\le 3n/c^{10}$ and $e_H(X\cup
 Y)\ge 3.8|X|>1.2|X\cup Y|$ --- a contradiction.

For the sets $X$ of constant size $|X|=x\le C_0$, if
$|N_H(X)\setminus X|<2|X|$, then there is a set $Y$ of cardinality $|Y|=2x-1$
with $N_H(X)\setminus X\subseteq Y$. Since all degrees in $H$ are at least 5,
and $X$ spans at most $x+2$ edges, it follows that $e_H(X\cup Y)\ge
5x-e(X)\ge 4x-2$. The probability that there exists a pair
$X,Y$ with $|X|=x$ in $\cG(n,p)$ is bounded from above by
\[
\sum_{x\leq C_0} \binom{n}{x}\binom{n}{2x-1}\binom{\binom{3x-1}{2}}{4x-2}p^{4x-2}
=O\left(n^{3x-1}p^{4x-2}\right) \leq n^{-x+1+o(1)}=o(n^{-2})
\]
for $x\ge 4$. For $x=3$, a simple union bound shows that the
probability of the existence of a pair $X,Y$ as above is at most
$n^{-2x-1+\binom{x}{2}+o(1)}=o(n^{-2})$ and for $x=1,2$ the
sought claim follows from the minimum degree assumption.
\end{proof}

We will also need that typically in the supercritical regime the
$k$-core contains a proportion of vertices rapidly approaching 1.

\begin{lemma}\label{lem-coresize}
Let $G\sim \cG(n,p)$ for $p=c/n$ with $k+99\sqrt{k\log k}\le
c(n)\le 2\log n$. If $k$ is large enough, then with probability
$1-o(n^{-2})$ the $k$-core satisfies $|\Gk[G]|\ge
n\left(1-c^{-200}\right)$.
\end{lemma}

Of course, there is nothing new in this lemma, and the estimate on
the size of the $k$-core can easily be improved. We present its
(straightforward) proof here mainly for the sake of completeness.

\begin{proof}
Let
\[
t=n / c^{200}\, .
\]
Construct the (possibly empty) $k$-core by peeling off repeatedly
vertices of degree less than $k$. If this process lasts for $t$
steps, then $G$ contains a subset $V_0$ of $|V_0|=t$ vertices such
that all degrees from $V_0$ to $V-V_0$ are less than $k$. For a
given $V_0$ and a given $v\in V_0$, the degree of $v$ to $V-V_0$ in
$\cG(n,p)$ is $\bin(n-t,p)$, and by Chernoff's inequality (see, e.g., \cite{AS}),
\begin{eqnarray*}
\rho&=&\P(\bin(n-t,p)<k)\le
\exp\left(-\frac{((n-t)p-k)^2}{2(n-t)p}\right) \le
\exp\left(-\frac{((n-t)p-k)^2}{2np}\right)\\
&\le&\exp\left(-\frac{\left(c-c^{-199}-k\right)^2}{2c}\right)\,.
\end{eqnarray*}
One can verify easily that for $c\ge k+99\sqrt{k\log k}$ and large $k$ it holds
that
\[
\frac{\left(c-c^{-199}-k\right)^2}{2c}\ge 300\log c\,,
\]
 and thus $\rho\le c^{-300}$. Hence the probability of the
 existence of a set $V_0$ as above is at most
\[
 \binom{n}{t}\rho^t\le \left(\frac{en}{t}\right)^t\rho^t\le
 (ec^{-100})^t=o(n^{-2})
\]
 for large enough $c$ (or large enough $k$, as $c\ge k$). If no such
 set $V_0$ exists, then clearly the $k$-core of $G$ has at least
 $n-t$ vertices, as required.
 \end{proof}

For the critical regime, the analog of Lemma \ref{lem-critical} is
now the following statement.

\begin{lemma}\label{lem-critical2}
Fix $k$ large enough and let $G\sim \cG(n,m)$ with
$m=cn/2$ and $k\leq c \leq k+100\sqrt{k\log k}$.
Then with probability $1-o(n^{-2})$, for any $G'=(V',E')\subset G$ with
minimum degree $\delta(G')\geq k$ such that $V'=V(\Gk[G])$ and $E(\Gk[G]) \setminus E'$
contains at most $n/\log\log n$ edges, the graph
$G'$ contains a $(2k_1)$-factor.
\end{lemma}

The proof of this lemma can be derived by repeating the arguments of
Pra{\l}at, Verstra{\"e}te and Wormald~\cite{PVW} and of Chan and
Molloy~\cite{CM}. They proved that for large enough $k$, upon
creation, the $k$-core $\Gk[G]$ with high probability either
contains a $(k-2)$-factor or is $(k-2)$-factor-critical~\cite{PVW},
or even contains a $(k-1)$-factor or is $(k-1)$-factor-critical~\cite{CM}. (A graph $G$ is $t$-factor-critical if for every $v\in
V(G)$, the subgraph $G-v$ contains a $t$-factor.) Both these papers
use Tutte's factor theorem to derive the likely existence of a
required factor. As one can anticipate (and we indeed verified),
altering  the $k$-core slightly (by deleting at most $n/\log\log n$
edges) while preserving minimum degree $k$, do not influence the
proof. We would also like to point out that the results in both
papers~\cites{CM,PVW} used the assumption that $c\leq k+2\sqrt{k\log
k}$, but one can easily check that replacing $2$ with any other
constant will still work. Finally note that $2k_1<k-2$ and either $k-2$ or $k-1$ is even. Thus, Lemma~\ref{lem-critical2} follows from the above results by the well-known
fact that a $2s$-factor contains a $2t$-factor for all $t \leq s$.

\begin{proof}[\textbf{\emph{Proof of Theorem~\ref{thm-pack-small-c} (sketch)}}]
We utilize the same approach, notation and terminology as in the
proof of Theorem~\ref{thm-ham-small-c}. We start by summarizing all
the results that have been established in the proof of that theorem
which we use here.

Let $\cS=(e_1,\ldots,e_m)$ be a uniformly chosen ordered subset of
$m$ edges out of the $\binom{n}2$ possible ones on the vertex set
$V=[n]$, $m' = m - n/\log\log n$, and let $G$ and $H$ be the
graphs on the vertex set $V$ with edge sets
$E(G)=\{e_1,\ldots,e_m\}$, $E(H)=\{e_1,\ldots,e_{m'}\}$. As
$G\sim\cG(n,m)$, if $\Gk$ is empty then there is nothing left to
prove. We therefore assume otherwise and aim to establish that $\Gk$
contains a family of $k_1$ edge-disjoint Hamilton cycles with
probability $1-o(n^{-2})$. Given the above two graphs $H \subset G$
on the same vertex set $V$, define the partition of the vertex set
of $\Gk$, the $k$-core of $G$, into $\cB_k=\cB_k(G,H)$ and $\cC_k
=\cC_k(G,H)$ so that $\cB_k$ consists of all vertices of $\Gk$
having at least $k$ neighbors in $V(\Gk)$ already in the subgraph
$H$ and $\cC_k = V(\Gk) \setminus \cB_k$. Denote the indices of the
subset of the final $m-m'$ edges having both endpoints in $\cB_k$ by
$\cJ$ and let $J = |\cJ|$. Recall that we have already established
that upon conditioning on $\cB_k$, $\cC_k$ (such that
$\cB_k\cup\cC_k \neq\emptyset$ by our assumption) as well as $\cJ$
and all the edges $\cS^* = \{ e_j : j \notin \cJ\}$, the remaining
$J$ edges of $\cS$ are uniformly distributed over all edges missing
from $\cB_k$ (that is, edges with both endpoints in $\cB_k$ that did
not appear among $e_1,\ldots,e_{m'}$). Moreover we also proved that
$J \geq n/(100\log \log n)$ with probability $1-O(n^{-9})$.

Letting $F=(V,\cS^*)$, as in the proof of
Theorem~\ref{thm-ham-small-c} we have $V(\Gk[F])=V(\Gk[G])$. Further
note that $G'=\Gk[F]$ satisfies the conditions of Lemma
\ref{lem-critical2} and therefore contains a $(2k_1)$-factor.
By a well-known fact from graph theory, this $(2k_1)$-factor can
be decomposed into $k_1$ edge-disjoint 2-factors
$\Lambda_1,\ldots,\Lambda_{k_1}$. Observe that, due to Claim
\ref{clm-max-disjoint-cycles}, with probability $1-o(n^{-2})$ each
2-factor $\Lambda_i$ has at most $n/\sqrt{\log n}$ cycles. We have
at our disposal a set $\cJ$ of $J\ge n/(100\log\log n)$ random edges
to fall into the set $\cB_k$. Split $\cJ$ into $k_1$ sets
${\cJ}_1,\ldots,{\cJ}_{k_1}$ of nearly equal  sizes $|{\cJ}_i|\ge
\lfloor J/k_1\rfloor$ and use ${\cJ}_i$ to convert the $i$-th
2-factor $\Lambda_i$ into a Hamilton cycle, of course relying on the
existing edges as the backbone. More specifically, we repeat $k_1$
iterations of the following process. Suppose we have already
constructed  Hamilton cycles $H_1,\ldots,H_{i-1}$. Our goal is to
convert $\Lambda_i$ into a Hamilton cycle $H_i$ using the edges of
\[
G_i=G'-(H_1\cup\ldots\cup
H_{i-1})\cup(\Lambda_{i+1}\cup\ldots\cup\Lambda_{k_1})
\]
 and the random edges from ${\cJ}_i$. Observe that the minimum degree of the
relevant graph satisfies $\delta(G_i)\ge k-(2k_1-2)\ge 5$, thus
making Lemma~\ref{lem-expansion} applicable. By this lemma, all
connected components of $G_i$ are of size at least $n/c^{10}$ and
therefore $G_i$ has at most $c^{10}$ connected components. Moreover,
since $|\cC_k| \leq 2(m-m')=2n/(\log \log n)$, at least
$n/c^{10}-|\cC_k| \geq n/c^{11}$ vertices of each component belong
to $\cB_k$. Thus, adding any random edge in $\cB_k$ will connect two
of these components with probability at least $c^{-22}$. By standard concentration arguments, it is then easy to see that throwing
in, say, $\log ^2 n$ edges from ${\cJ}_i$ will turn $G_i$ into
a connected graph with probability $1-o(n^{-2})$.

Now we apply Lemma~\ref{lem-Posa}. Due to this lemma together with
the expansion property guaranteed by Lemma~\ref{lem-expansion}, at
any step $G_i$ contains at least $n/(2c^{20})$ boosters.
Moreover, the number of such boosters with both endpoints in $\cB_k$
is at least $n/(2c^{20})-n |\cC_k|\geq
n/c^{21}$. Thus, the probability that a random edge falling
into $\cB_k$ hits a booster is at least $c^{-21}$. Since $\Lambda_i$
has only at most $n/\sqrt{\log n}$ cycles to begin with, hitting
$n/\sqrt{\log n}+1$ such boosters will transform it into a Hamilton
cycle. Also note that we have not yet used the vast majority of the
edges from ${\cJ}_i$, and can still add $|{\cJ}_i|-\log ^2
n>n/(200\log \log n)$ random edges to $G_i$. By adding all
these edges, we collect a number of boosters which
stochastically dominates a binomial random variable with
parameters $n/(200\log \log n)$ and $c^{-21}$. Therefore, by
standard concentration arguments, we create the next Hamilton cycle
$H_i$ with probability at least $1-o(n^{-2})$. Finally, we release
the edges of $\Lambda_i$ not used in $H_i$ back into the graph,
guaranteeing that in the next iteration too the minimum degree of
the backbone graph is at least 5.
\end{proof}

\begin{proof}[\textbf{\emph{Proof of Theorem~\ref{thm-pack-large-c}.}}]
Observe first that due to Lemma~\ref{lem-coresize}, the difference
in sizes of the $k$-cores $\Gk[G]$ and $\Gk[H]$ is relatively small:
\[
|V(\Gk[G])|-|V(\Gk[H])|\le n-|V(\Gk[H])|\le n/c^{200}\,.
\]
Let
$V'=V(\Gk[H])$ and let $R$ be the set of $n$ random edges. For every
edge $e=(u,v) \in R$ (uniformly distributed over missing edges in
$H$) expose whether both $u,v$ are in $V'$ or not. If at least one
of them is not in $V'$ expose both endpoints of $e$. Otherwise put
$e$ into $R'$ and note that $R'$ is the set of random edges
uniformly distributed over all missing edges in the induced graph of
$H$ on $V'$. Also note that since $|V'| \geq (1-c^{-200})n$, the number
of edges in $R'$ is stochastically dominated by the binomial random
variable with parameters $n$ and $2/3$. Therefore with probability
at least $1-o(n^{-2})$ in the end of this process $R'$ has at least
$n/2$ edges. Moreover, this process reveals the set of vertices of
the $k$-core of the graph $G$ obtained from $H$ by adding edges in
$R$.

Fix a collection $\Lambda_1,\ldots,\Lambda_{k_1}$ of edge-disjoint
Hamilton cycles in $\Gk[H]$. We will use the set $R'$ of random
edges to convert them into edge-disjoint Hamilton cycles in
$\Gk[G]$. To this end, split the set  $R'$ in $k_1$ parts
$R_1,\ldots, R_{k_1}$ of sizes $|R_i|\ge n/(3k_1)$. We perform $k_1$
iterations of the following process. Suppose that for $i\in [k_1]$
we have already created Hamilton cycles $H_1,\ldots,H_{i-1}$ in
$\Gk[G]$. Let
\[
G_i=\Gk[G]-(H_1\cup\ldots\cup
H_{i-1})\cup(\Lambda_{i+1}\cup\ldots\cup\Lambda_{k_1})\,.
\]
By
definition, the minimum degree of $G_i$ is at least 5, enabling us
to utilize Lemma~\ref{lem-expansion}. Due to this lemma, all
connected components of $G_i$ are of size at least $n/c^{10}$ and
therefore $G_i$ has at most $c^{10}$ connected components. Moreover,
at least $n/c^{10}-n/c^{200} \geq n/c^{20}$ vertices of each
component belong to $V'$. Thus adding any random edge in $V'$ will
connect two of these components with probability at least $c^{-40}$.
Therefore, by standard concentration arguments, it is easily seen
that throwing in (say) first $\log ^2 n$ edges from $R_i$ will turn
$G_i$ into a connected graph with probability $1-o(n^{-2})$.

Next, as in the proof of Theorem~\ref{thm-ham-large-c}, we
add to $\Lambda_i$ all vertices of $V(\Gk[G])-V(\Gk[H])$ (at most
$n/c^{200}$ of them) to have a 2-factor with at most $n/c^{200}+1$
cycles. We merge these cycles into one gradually, using the
remaining $|R_i| - \log^2 n \geq n/(4k_1)>n/c^2$ of the edges of the
$R_i$. By the expansion properties (Lemma~\ref{lem-expansion}) of $G_i$, we can derive from Lemma~\ref{lem-Posa} that there
are at least $n^2/(2c^{20})$ boosters in $G_i$. Moreover, all but
$n^2/c^{200}$ such boosters have both endpoints in $V'$.
Hence, the probability that a new edge, uniform over all missing
edges in $V'$, is a booster is at least $c^{-21}$. Hitting
$n/c^{200}+2$ such boosters will transform our 2-factor into a
Hamilton cycle in $G_i$. Since the number of boosters that we encounter is
stochastically dominated by a binomial variable $\bin(n/c^2, c^{-21})$, we deduce that with probability
$1-o(n^{-2})$ exposing random edges from $R_i$ will create the next
Hamilton cycle $H_i$ in $\Gk[G]$. Finally, we release the edges of
$\Lambda_i$ not used in $H_i$ back into the graph, guaranteeing that
in the next iteration too the minimum degree of the backbone graph
would be at least 5.
 \end{proof}

\section{Open Problems}\label{sec:open}
The main problem that remains open is, of course, whether the result of Theorem~\ref{thm-1} (the Hamiltonicity of $\Gk_t$ for all $t\geq \tau_k$ w.h.p.\ provided $k\geq 15$) can be pushed all the way down to $k\geq3$.
The bottleneck in our proof, as mentioned before, was finding a 2-factor in the $k$-core minus the unrevealed sprinkling edges. Due to the complexity of this graph, instead of appealing to Tutte's criterion for a 2-factor, we construct one using 2 disjoint perfect matchings. This appears to be possible at $\tau_k$ as long as $k\geq 7$. By taking $k\geq 15$ we were able both to simplify the arguments used to recover this 2-factor, as well as to guarantee that $\Gk_t$ would remain Hamiltonian for all $t\geq \tau_k$ w.h.p. It is thus quite plausible that, when paired with a more refined analysis, the present framework should produce an improved bound on $k$. However, a new idea seems to be needed to take it down to $k=3$.

For the packing problem, recalling the conjecture
of Bollob\'as, Cooper, Fenner and Frieze~\cite{BCFF} that the $k$-core w.h.p.\ contains $\lfloor (k-1)/2\rfloor$ edge-disjoint Hamilton cycles, our results here are only one Hamilton cycle short of this conjecture. However, the
present methods appear to be incapable of bridging this small gap.
In particular, in our cycle-by-cycle extraction argument it is
crucial to have minimum degree 5 at every stage, so as to guarantee
sufficient expansion of small sets (as given by Lemma~\ref{lem-expansion} and used by Lemma~\ref{lem-Posa}).
Once we extract $\lfloor (k-3)/2\rfloor$ Hamilton cycles from the $k$-core, we will be left with a subgraph of minimum degree 3,
and this by itself is not enough to ensure the required expansion
(e.g., consider a set of degree-3 vertices forming a cycle). 
In conclusion, while new ideas are needed to fully settle
this conjecture, our result provides a fairly firm evidence
for its validity.

\vspace{0.2cm} \noindent {\bf Acknowledgment.}\, A major part of
this work was carried out when the first and the third authors were
visiting Microsoft Research at Redmond, WA. They would like to thank
the Theory Group at Microsoft Research for hospitality and for
creating a stimulating research environment.
The authors are also grateful to an anonymous referee for useful comments and corrections.


\begin{bibdiv}
\begin{biblist}[\normalsize]

\bib{AKS}{article}{
   author={Ajtai, M.},
   author={Koml{\'o}s, J.},
   author={Szemer{\'e}di, E.},
   title={First occurrence of Hamilton cycles in random graphs},
   conference={
      title={Cycles in graphs},
      address={Burnaby, B.C.},
      date={1982},
   },
   book={
      series={North-Holland Math. Stud.},
      volume={115},
      publisher={North-Holland},
      place={Amsterdam},
   },
   date={1985},
   pages={173--178},
}


\bib{AS}{book}{
  author={Alon, Noga},
  author={Spencer, Joel H.},
  title={The probabilistic method},
  edition={3},
  publisher={John Wiley \& Sons Inc.},
  date={2008},
  pages={xviii+352},
}

\bib{Bennett}{article}{
   title = {Probability inequalities for the sum of independent random Variables},
   author = {Bennett, George},
   journal = {Journal of the American Statistical Association},
   volume = {57},
   number = {297},
   pages = {33--45},
   date = {1962},
}

\bib{Bollobas-1}{article}{
   author={Bollob{\'a}s, B{\'e}la},
   title={The evolution of sparse graphs},
   conference={
      title={Graph theory and combinatorics},
      address={Cambridge},
      date={1983},
   },
   book={
      publisher={Academic Press},
      place={London},
   },
   date={1984},
   pages={35--57},
}

\bib{Bollobas}{book}{
   author={Bollob\'as, Bela},
   title={Random graphs},
   edition={2},
   publisher={Cambridge University Press, Cambridge},
   date={2001},
   pages={xviii+498},
}


\bib{BCFF}{article}{
   author={Bollob{\'a}s, B.},
   author={Cooper, C.},
   author={Fenner, T. I.},
   author={Frieze, A. M.},
   title={Edge disjoint Hamilton cycles in sparse random graphs of minimum degree at least $k$},
   journal={J. Graph Theory},
   volume={34},
   number={1},
   date={2000},
   pages={42--59},
}

\bib{BFF}{article}{
  author={Bollob{\'a}s, B.},
   author={Fenner, T. I.},
   author={Frieze, A. M.},
   title={Hamilton cycles in random graphs of minimal degree at least $k$},
   conference={
      title={A tribute to Paul Erd\H os},
   },
   book={
      publisher={Cambridge Univ. Press},
      place={Cambridge},
   },
   date={1990},
   pages={59--95},
}

\bib{BolF}{article}{
   author={Bollob{\'a}s, B{\'e}la},
   author={Frieze, Alan M.},
   title={On matchings and Hamiltonian cycles in random graphs},
   conference={
      title={Random graphs '83},
      address={Pozna\'n},
     date={1983},
   },
   book={
     series={North-Holland Math. Stud.},
      volume={118},
      publisher={North-Holland},
      place={Amsterdam},
   },
   date={1985},
   pages={23--46},
}

\bib{CW}{article}{
   author={Cain, Julie},
   author={Wormald, Nicholas},
   title={Encores on cores},
   journal={Electron. J. Combin.},
   volume={13},
   date={2006},
   number={1},
   pages={Research Paper 81, 13 pp. (electronic)},
}

\bib{CM}{article}{
  title={(k+ 1)-Cores Have k-Factors},
  author={Chan, Sio On},
  author={Molloy, Michael},
  journal={Combinatorics, Probability and Computing},
  volume={21},
  number={6},
  date={2012},
  pages={882--896},
}

\bib{Chvatal}{article}{
   author={Chv{\'a}tal, V.},
   title={Almost all graphs with $1.44n$ edges are $3$-colorable},
   journal={Random Structures Algorithms},
   volume={2},
   date={1991},
   number={1},
   pages={11--28},
}

\bib{Cooper}{article}{
   author={Cooper, Colin},
   title={The cores of random hypergraphs with a given degree sequence},
   journal={Random Structures Algorithms},
   volume={25},
   date={2004},
   number={4},
   pages={353--375},
}

\bib{ER60}{article}{
   author={Erd{\H{o}}s, P.},
   author={R{\'e}nyi, A.},
   title={On the evolution of random graphs},
   journal={Magyar Tud. Akad. Mat. Kutat\'o Int. K\"ozl.},
   volume={5},
   date={1960},
   pages={17--61},
}


\bib{Frieze}{article}{
    author = {Frieze, Alan},
    title = {On a greedy 2-matching algorithm and Hamilton cycles in random graphs with minimum degree at least three},
    journal = {Random Structures Algorithms},
    status = {to appear},
}



\bib{Hoeffding}{article}{
   author={Hoeffding, Wassily},
   title={Probability inequalities for sums of bounded random variables},
   journal={J. Amer. Statist. Assoc.},
   volume={58},
   date={1963},
   pages={13--30},
}

\bib{JL07}{article}{
   author={Janson, Svante},
   author={Luczak, Malwina J.},
   title={A simple solution to the $k$-core problem},
   journal={Random Structures Algorithms},
   volume={30},
   date={2007},
   number={1--2},
   pages={50--62},
}

\bib{JL08}{article}{
   author={Janson, Svante},
   author={Luczak, Malwina J.},
   title={Asymptotic normality of the $k$-core in random graphs},
   journal={Ann. Appl. Probab.},
   volume={18},
   date={2008},
   number={3},
   pages={1085--1137},
}

\bib{JLR}{book}{
   author={Janson, Svante},
   author={{\L}uczak, Tomasz},
   author={Rucinski, Andrzej},
   title={Random graphs},
   series={Wiley-Interscience Series in Discrete Mathematics and
   Optimization},
   publisher={Wiley-Interscience, New York},
   date={2000},
   pages={xii+333},
}


\bib{KS}{article}{
   author={Koml{\'o}s, J{\'a}nos},
   author={Szemer{\'e}di, Endre},
   title={Limit distribution for the existence of Hamiltonian cycles in a random graph},
   journal={Discrete Math.},
   volume={43},
   date={1983},
   number={1},
   pages={55--63},
}

\bib{Korshunov}{article}{
  author={Korshunov, A.},
  title={Solution of a problem of Erd\H{o}s and R\'enyi on Hamilton
  cycles in non-oriented graphs},
  journal={Soviet Math. Dokl.},
  volume={17},
  pages={760--764},
}

\bib{KLS}{article}{
   author={Krivelevich, Michael},
   author={Lubetzky, Eyal},
   author={Sudakov, Benny},
   title={Hamiltonicity thresholds in Achlioptas processes},
   journal={Random Structures Algorithms},
   volume={37},
   date={2010},
   number={1},
   pages={1--24},
}


\bib{LP}{book}{
   author={Lov{\'a}sz, L{\'a}szl{\'o}},
   author={Plummer, Michael D.},
   title={Matching theory},
   publisher={AMS Chelsea Publishing, Providence, RI},
   date={2009},
   pages={xxxiv+554},
}

\bib{Luczak87}{article}{
   author={{\L}uczak, Tomasz},
   title={On matchings and Hamiltonian cycles in subgraphs of random graphs},
   conference={
      title={Random graphs '85},
      address={Pozna\'n},
      date={1985},
   },
   book={
      series={North-Holland Math. Stud.},
      volume={144},
      publisher={North-Holland},
      place={Amsterdam},
   },
   date={1987},
   pages={171--185},
}

\bib{Luczak91}{article}{
   author={{\L}uczak, Tomasz},
   title={Size and connectivity of the $k$-core of a random graph},
   journal={Discrete Math.},
   volume={91},
   date={1991},
   number={1},
   pages={61--68},
}

\bib{Luczak92}{article}{
   author={{\L}uczak, Tomasz},
   title={Sparse random graphs with a given degree sequence},
   conference={
      title={Random graphs, Vol.\ 2},
      address={Pozna\'n},
      date={1989},
   },
   book={
      series={Wiley-Intersci. Publ.},
      publisher={Wiley},
      place={New York},
   },
   date={1992},
   pages={165--182},
}

\bib{Molloy96}{article}{
   author={Molloy, Michael},
   title={A gap between the appearances of a $k$-core and a
   $(k+1)$-chromatic graph},
   journal={Random Structures Algorithms},
   volume={8},
   date={1996},
   number={2},
   pages={159--160},
}

\bib{Molloy}{article}{
   author={Molloy, Michael},
   title={Cores in random hypergraphs and Boolean formulas},
   journal={Random Structures Algorithms},
   volume={27},
   date={2005},
   number={1},
   pages={124--135},
}

\bib{PSW}{article}{
   author={Pittel, Boris},
   author={Spencer, Joel},
   author={Wormald, Nicholas},
   title={Sudden emergence of a giant $k$-core in a random graph},
   journal={J. Combin. Theory Ser. B},
   volume={67},
   date={1996},
   number={1},
   pages={111--151},
}

\bib{Posa}{article}{
   author={P{\'o}sa, L.},
   title={Hamiltonian circuits in random graphs},
   journal={Discrete Math.},
   volume={14},
   date={1976},
   number={4},
   pages={359--364},
}

\bib{PVW}{article}{
   author={Pra{\l}at, Pawe{\l}},
   author={Verstra{\"e}te, Jacques},
   author={Wormald, Nicholas},
   title={On the threshold for $k$-regular subgraphs of random graphs},
   journal={Combinatorica},
   volume={31},
   date={2011},
   number={5},
   pages={565--581},
}


\bib{Riordan}{article}{
   author={Riordan, Oliver},
   title={The $k$-core and branching processes},
   journal={Combin. Probab. Comput.},
   volume={17},
   date={2008},
   number={1},
   pages={111--136},
}

\bib{Tutte}{article}{
   author={Tutte, W. T.},
   title={The factorization of linear graphs},
   journal={J. London Math. Soc.},
   volume={22},
   date={1947},
   pages={107--111},
}

\end{biblist}
\end{bibdiv}
\end{document}